\documentclass[11pt]{amsart}

\usepackage[T1]{fontenc}
\usepackage[utf8]{inputenc}
\usepackage{lmodern}
\usepackage{amsmath,amssymb,amsfonts,mathtools}
\usepackage{amsthm}
\usepackage{enumitem}
\usepackage{microtype}
\usepackage{xcolor}
\usepackage{ifoddpage}
\usepackage[a4paper,left=2.5cm,right=2.5cm,top=2.7cm,bottom=2.7cm]{geometry}
\usepackage[colorlinks=true,linkcolor=blue,citecolor=blue,urlcolor=blue]{hyperref}

\numberwithin{equation}{section}
\newtheorem{theorem}{Theorem}[section]
\newtheorem{proposition}[theorem]{Proposition}
\newtheorem{lemma}[theorem]{Lemma}

\newtheorem{definition}[theorem]{Definition}
\theoremstyle{remark}
\newtheorem{remark}[theorem]{Remark}

\DeclareMathOperator{\spanop}{span}
\newcommand{\R}{\mathbb R}
\providecommand{\C}{}
\renewcommand{\C}{\mathbb C}
\newcommand{\cW}{\mathcal W}
\newcommand{\cQ}{\mathcal Q}
\newcommand{\cQt}{\widetilde{\mathcal Q}}
\newcommand{\cC}{\mathcal C}
\newcommand{\cE}{\mathcal E}
\newcommand{\cA}{\mathcal A}
\newcommand{\cI}{\mathcal I}
\newcommand{\cR}{\mathcal R}
\newcommand{\cU}{\mathcal U}
\newcommand{\cZ}{\mathcal Z}
\newcommand{\dd}{\mathrm d}
\DeclareMathOperator{\rank}{rank}

\title[Generic spectral determination with \(\mathbb Z_2\)-symmetry]{%
Generic Spectral Determination of Semiclassical Schr\"odinger Operators
with \(\mathbb Z_2\)-Symmetry
}

\author{Qiaoling Wei}
\address{School of Mathematical Sciences, Capital Normal University, China}
\email{wql03@cnu.edu.cn}
\date{}

\makeatletter
\@ifundefined{subjclassname@2020}{%
  \@namedef{subjclassname@2020}{\textup{2020} Mathematics Subject Classification}%
}{}
\makeatother

\begin{document}

\begin{abstract}
Consider the two-dimensional semiclassical Schrödinger operator
$P_\hbar=-\frac{\hbar^2}{2}\Delta+V$ on $\mathbb R^2$, where $V$ has a
nondegenerate well at the origin, its harmonic frequencies are rationally
independent, and $V$ is \(\mathbb Z_2\)-symmetric.  We prove that, generically, the
first three layers of the quantum Birkhoff normal form (QBNF) determine the
full Taylor series of $V$ at the origin, up to the unavoidable spatial inversion
$V(x)\mapsto V(-x)$.  The generic condition depends only on the jet of $V$
through order ten.  Consequently, for real analytic potentials with a unique
isolated global well at the origin, the low-lying semiclassical spectrum generically determines $V$
up to spatial inversion.
\end{abstract}

\maketitle

\section{Introduction}

Kac's question \emph{Can one hear the shape of a drum?} asks whether the
eigenfrequencies of a vibrating membrane determine its shape \cite{Kac66}.
It is the most familiar instance of the inverse spectral problem: to what
extent does the spectrum of an operator determine the geometric or physical
object that defines it?  As emphasized in Zelditch's survey
\cite{Zel14}, this question has three basic types. For the Dirichlet Laplacian on a bounded planar domain, the unknown is the boundary of the domain.  For the Laplace--Beltrami operator on a closed manifold, the unknown
is the Riemannian metric, up to isometry.  For a semi-classicalSchr\"odinger operator , the unknown is the potential.


Positive results usually rely on hypotheses that control the geometry and dynamics, including nondegeneracy, symmetry, and convexity assumptions. The first two types of the inverse problem provide important precedents.  For
analytic plane domains, Zelditch proved spectral determination within a class
having one isometric involution that reverses a nondegenerate bouncing-ball
orbit \cite{Zel09}.  De~Simoi--Kaloshin--Wei established dynamical spectral
rigidity for finitely smooth \(\mathbb Z_2\)-symmetric strictly convex domains
close to a circle \cite{DSKW17}; more recently, Hezari--Zelditch proved that
ellipses of sufficiently small eccentricity are spectrally unique among all
smooth plane domains \cite{HZ22}. 
For a closed negatively curved surface, the corresponding dynamical problem is much more rigid: Otal proved that the metric is determined, up to isometry, by its marked length spectrum \cite{Otal90}.

The present paper concerns the third type.  We consider the semiclassical
Schr\"odinger operator
\[
 P_\hbar=-\frac{\hbar^2}{2}\Delta+V,
\]
and ask whether the spectrum of the operator can determine the potential \(V\).  There are two distinct
versions of this problem: one may fix \(\hbar\), or use the
one-parameter family of spectra as \(\hbar\to0\).  The latter, which is the
setting here, carries substantially more information and is correspondingly
more tractable. 

A natural intermediary between the spectral data and the potential is the
quantum Birkhoff normal form (QBNF).  There is a filtered Moyal automorphism
transforming the Weyl symbol of \(P_\hbar\) to its unique resonant QBNF, whose
classical leading part represents the formal symplectic equivalence class of
the Hamiltonian dynamics.  In the nonresonant bottom-of-the-well setting, the
complete QBNF and the semiclassical asymptotics of the low-lying eigenvalues
determine one another \cite{GPU07,CdV09}.  The inverse spectral problem
therefore separates into recovering the QBNF from the spectrum and recovering
the potential from the QBNF.  


In dimension one, classical inverse Sturm--Liouville theory illustrates the
role of symmetry: one spectrum is insufficient in general, whereas two
spectra---or one under midpoint reflection symmetry---suffice
\cite{Borg46,Lev49,Mar52}.  Semiclassical reconstruction results under generic
hypotheses were obtained from spectral invariants \cite{Hez09,CdV11,GW12}; at
the QBNF level, the first two layers determine the
Taylor series after fixing the sign of the nonzero cubic coefficient \cite{CdVG11}.  By
contrast, without suitable assumptions, even the full semiclassical spectral
asymptotics may fail to determine the potential \cite{GH12,West23}.

In dimensions greater than one, known bottom-of-the-well reconstruction
results impose symmetry.  For potentials
invariant under all coordinate reflections, Guillemin--Uribe showed, through
the classical Birkhoff normal form, that the low-lying spectrum determines the
Taylor series at the well \cite{GU07}. Hezari derived
explicit wave invariants and obtained reconstruction for
potentials of the form
\(V(x)=f(x_1^2,\ldots,x_n^2)+x_n^3g(x_1^2,\ldots,x_n^2)\) \cite{Hez09}.  In the two-dimensional case,
Guillemin--Uribe used the first two QBNF layers to obtain reconstruction under
one coordinate reflection together with an additional symmetry \cite{GU11}.  With a single reflection symmetry, Wang
used the first two QBNF layers to recover the Taylor series once the sign of
the leading cubic coefficient and a full transverse family are prescribed
\cite{Wang24, Wang26}.  Separate determination results under the substantially
stronger assumption of radial symmetry proceed through semiclassical trace
invariants; we do not discuss that literature here. 

The problem addressed in the current paper is whether the QBNF itself determines the Taylor series of a two-dimensional potential with a single reflection symmetry. Zelditch's billiard result \cite{Zel09} provides a useful comparison. In that setting, the classical Birkhoff normal form of the first return map does not contain sufficient information to determine a domain with only one reflection symmetry, so that a normal-form approach would require the full quantum Birkhoff normal form of the associated quantum monodromy operator. Zelditch instead recovers the boundary Taylor coefficients directly from the Balian--Bloch invariants. A parallel phenomenon occurs in the Schrödinger settings: we prove that the first two QBNF layers leave a one-dimensional affine defect in the reconstruction. In contrast to the billiard argument, this defect can be detected within the QBNF itself by the \(\hbar^4\)-layer. We thus obtain an affirmative answer, up to the unavoidable spatial inversion, under a generic condition depending only on a finite jet.


\subsection{Statement of the main results}

We first formulate the reconstruction at the level of Taylor series.  Let
\(V_{\mathrm{phys}}\) denote the Taylor series at the origin of the potential
in the standard Schr\"odinger operator.  Assume that its constant and linear
terms vanish and that its quadratic part is positive definite.  We impose the
\(\mathbb Z_2\)-symmetry
\[
 V_{\mathrm{phys}}(y_1,y_2)=V_{\mathrm{phys}}(y_1,-y_2),
\]
take \(y_1\) along the reflection axis and \(y_2\) in the normal direction,
and write its quadratic part as
\[
 (V_{\mathrm{phys}})_2(y)
 =\frac12\bigl(v_1^2y_1^2+v_2^2y_2^2\bigr),
 \qquad v_1,v_2>0.
\]
The \(\mathbb Z_2\)-equivariant canonical dilation
\[
 y_j=\frac{x_j}{\sqrt{v_j}},
 \qquad \eta_j=\sqrt{v_j}\,\xi_j,
 \qquad j=1,2,
\]
transforms the physical Hamiltonian
\(\frac12|\eta|^2+V_{\mathrm{phys}}(y)\) into
\begin{equation}\label{eq:H-main}
 H_V(x,\xi)
 :=\frac12\bigl(v_1\xi_1^2+v_2\xi_2^2\bigr)+V(x)
 =H_2+\sum_{m\ge3}V_m(x),\qquad
 H_2=\frac12(v_1\Omega_1+v_2\Omega_2),\qquad
 \Omega_j=x_j^2+\xi_j^2.
\end{equation}
where
\(
 V(x):=V_{\mathrm{phys}}
 \left(\frac{x_1}{\sqrt{v_1}},\frac{x_2}{\sqrt{v_2}}\right)
\), and \(V_m\) is homogeneous of degree \(m\) in \(x\).
Throughout the formal and analytic reconstruction below, \(V\) denotes this
normalized potential and 
\(P_\hbar^V:=P_\hbar^{V_{\mathrm{phys}}}=-\frac{\hbar^2}{2}\Delta_y+V_{\mathrm{phys}}(y)\)
denote the corresponding physical Schrödinger operator.

 The
\(\mathbb Z_2\)-symmetry makes \(V_m\) even in \(x_2\). 
In the normalized coordinates, we assume
\begin{equation}\label{eq:Vm-expansion}
 v_1,v_2>0,\quad v_1/v_2\notin \mathbb Q,\quad V_2=\frac{1}{2}(v_1x_1^2+v_2x_2^2),\quad 
 V_m(x_1,x_2)=\sum_{j+2k=m}a_{j,2k}x_1^jx_2^{2k}.
\end{equation}

Under the nonresonance condition in \eqref{eq:Vm-expansion}, the Hamiltonian \(H_V\) can be transformed, by a formal Moyal automorphism of the formal Weyl algebra in \(x,\xi,\hbar\), into a unique resonant quantum Birkhoff normal form (QBNF), which is a formal series in \(\hbar\) and the harmonic actions \(\Omega_1,\Omega_2\):
\begin{equation*}
 \mathcal B(H_V)=H_2+
 \sum_{2r+k+\ell\ge2}
 b_{r,k,\ell}\hbar^{2r}\Omega_1^k\Omega_2^\ell
 =H_2+\sum_{r\ge0}\hbar^{2r}B^{[2r]}(\Omega).
\end{equation*}

The quadratic spectral data determine the two frequencies only as an
unordered pair.  In the coordinates fixed above, we write $v_1$ for the
frequency along the reflection axis and $v_2$ for the frequency in the
normal direction.  We fix this labeled pair $v=(v_1,v_2)$ throughout and write $\mathcal B(V)$ for the QBNF of $H_V$.

Let \(\mathcal V_v\) be the space of all $\mathbb Z_2$-symmetric
formal potentials of the form \eqref{eq:Vm-expansion}.  It is an infinite
jet space, in which a potential is identified with its sequence of Taylor
coefficients.  We endow it with the cylinder topology.  With this topology, convergence in \(\mathcal V_v\) is coefficientwise, and \(\mathcal V_v\) is a Baire space.  For \(\sigma\in\{\pm1\}\), let
\[
 \mathcal V_v^\sigma
 :=\{V\in\mathcal V_v:\sigma a_{30}(V)>0\}.
\]
Each \(\mathcal V_v^\sigma\) carries the relative cylinder topology.

A property is called \emph{generic} if it holds on a dense
\(G_\delta\) subset of the indicated ambient space. Notice that, whenever the ambient space is a Baire space, every countable intersection of open dense subsets is dense.

For \(m\ge3\), let \(J_{m,v}\) be the finite dimensional affine space of
\(\mathbb Z_2\)-symmetric real jets
\[
 j^mV=(V_2,V_3,\ldots,V_m),
\]
with the quadratic component \(V_2\) fixed by \(v\).
 We set
\[
 J_{m,v}^\sigma:=\{X\in J_{m,v}:\sigma a_{30}(X)>0\}.
\]

\begin{theorem}\label{thm:main}
For each \(\sigma\in\{\pm1\}\), there exists a dense \(G_\delta\) subset
\(\mathfrak G_{10,v}^\sigma\subset J_{10,v}^\sigma\) such that, for every
\(V,W\in\mathcal V_v^\sigma\) with
\(j^{10}V\in\mathfrak G_{10,v}^\sigma\),
\[
 \bigl(B^{[0]},B^{[2]},B^{[4]}\bigr)(W)
 =
 \bigl(B^{[0]},B^{[2]},B^{[4]}\bigr)(V)
 \quad\Longrightarrow\quad W=V.
\]
Moreover, the complete Taylor series of 
\(V\in\mathcal V_v^\sigma\) is determined by
\((B^{[0]},B^{[2]},B^{[4]})(V)\).
In particular, 
the first three QBNF layers
 determine the potential within \(\mathcal V_v^\sigma\) generically.
\end{theorem}

Here and below, “determine” means recoverability from the prescribed QBNF data. The reconstruction first recovers \(a_{12}\) as the unique admissible value allowed by an explicit finite polynomial system; the
remaining Taylor coefficients are then recovered from the QBNF data by
recursive formulas.

The sign restriction in Theorem~\ref{thm:main} fixes the unavoidable spatial
inversion ambiguity.  Define the involution
\[
 \iota:\mathcal V_v\longrightarrow\mathcal V_v,
 \qquad
 (\iota V)(x):=V(-x).
\]
Then the QBNF is invariant under \(\iota\).
The sign-free form of the main result is therefore the following.

\begin{theorem}\label{thm:main-mod-symmetry}
There exists an \(\iota\)-invariant dense \(G_\delta\) subset
\(\mathfrak G_{10,v}\subset J_{10,v}\) such that, for every
\(V,W\in\mathcal V_v\) with \(j^{10}V\in\mathfrak G_{10,v}\),
\[
 \bigl(B^{[0]},B^{[2]},B^{[4]}\bigr)(W)
 =
 \bigl(B^{[0]},B^{[2]},B^{[4]}\bigr)(V)
 \quad\Longrightarrow\quad
 W=V\ \text{or}\ W=\iota V.
\]
Moreover, $ \bigl(B^{[0]},B^{[2]},B^{[4]}\bigr)(V)$ determine
\(V\) up to the spatial inversion \(V\mapsto \iota V\).
In particular, 
the
first three QBNF layers  determine the potential up to spatial inversion generically.

\end{theorem}

To formulate genericity in the global real analytic category, let
\(\mathcal A_v\) be the space of functions in \(C^\omega(\R^2;\R)\) whose
Taylor series at the origin is of the form \eqref{eq:Vm-expansion}, and set
\[
 \mathcal A_v^\sigma
 :=\{V\in\mathcal A_v:\sigma a_{30}(V)>0\},
 \qquad \sigma\in\{\pm1\}.
\]
We equip \(C^\omega(\R^2;\R)\), and hence these subspaces, with the
standard holomorphic-germ topology described in
Appendix~\ref{app:analytic-topology}.  In this topology the finite-jet maps
at the origin are continuous, open, and surjective. Thus we have

\begin{theorem}\label{thm:analytic-genericity}
For every \(\sigma\in\{\pm1\}\), there exists a dense \(G_\delta\) subset
\(\mathcal G_v^\sigma\subset\mathcal A_v^\sigma\) such that, for every
\(V\in\mathcal G_v^\sigma\) and \(W\in\mathcal A_v^\sigma\),
\[
 (B^{[0]},B^{[2]},B^{[4]})(W)
 =(B^{[0]},B^{[2]},B^{[4]})(V)
 \quad\Longrightarrow\quad W=V.
\]
Moreover, the first three QBNF layers
 determine the potential within \(\mathcal A_v^\sigma\) generically.

\end{theorem}

We next pass from QBNF data to spectral data.   Assume that $V$ has a unique nondegenerate global minimum at the
origin and that the well is isolated at low energy:
\begin{equation}\label{eq:isolate}
 V^{-1}([0,\varepsilon])\text{ is compact for some }\varepsilon>0.
\end{equation}

This condition\footnote{For the Schr\"odinger Hamiltonian
\(H_V(x,\xi)\), the compact-sublevel condition is
equivalent to \(\liminf_{|(x,\xi)|\to\infty}H_V(x,\xi)>0\), which is the
bottom-of-the-well condition assumed by Colin de Verdi\`ere in \cite{CdV09}.}
ensures that the spectrum in \([0,\varepsilon)\) consists of isolated
eigenvalues of finite multiplicity.  For fixed \(N\geq1\) and 
sufficiently small \(\hbar>0\), let \(\lambda_N^V(\hbar)\) denote the
\(N\)-th eigenvalue of \(P_\hbar^V\) in \([0,\varepsilon)\), counted with
multiplicity.  Following \cite{CdV09}, define the semiclassical spectrum  to be the sequence
\[
 \operatorname{Spec}_{\mathrm{sc}}(P_\hbar^V)
 :=\bigl([\lambda_N^V(\hbar)]_{O(\hbar^\infty)}\bigr)_{N\ge1}.
\]
Thus equality of two semiclassical spectra means
\(\lambda_N^V(\hbar)-\lambda_N^W(\hbar)=O(\hbar^\infty)\) for every fixed
\(N\), with no estimate uniform in \(N\).

Under the above bottom-of-the-well assumptions and the nonresonance of the
harmonic frequencies, the semiclassical spectrum and the QBNF\footnote{In the
conventions of this paper, \(\mathcal B(V)\) is the formal full Weyl symbol of
the operator QBNF used in the spectral literature.  Its Weyl quantization
\(\operatorname{Op}_{\hbar}^{w}(\mathcal B(V))\), rewritten as a formal series
in the commuting harmonic-oscillator actions, is the semiclassical Birkhoff
normal form of \cite{CVN08,CdV09,GPU07}; the symbol and operator descriptions
determine one another.} determine one another.  The QBNF determines the
complete asymptotic expansion of every fixed low-lying eigenvalue
\cite{CVN08}, while the converse statement---that their ordered collection
determines the complete QBNF---is proved by Colin de Verdi\`ere \cite{CdV09};
see also Guillemin--Paul--Uribe \cite{GPU07}.

For comparison, fix \(0<\delta<\varepsilon\) and define the exact
spectrum in the low energy window by
\[
 \operatorname{Spec}_{\delta}(P_\hbar^V)
 :=\bigl(\lambda_N^V(\hbar)\bigr)_{
 N:\,0\leq\lambda_N^V(\hbar)\leq\delta}.
\]
By definition, this
fixed-window family determines $\operatorname{Spec}_{\mathrm{sc}}(P_\hbar^V)$.

Let
\(\mathcal A_v^{\mathrm{trap}}\subset\mathcal A_v\) be the class of
potentials for which the origin is the unique global minimum and for which
\(V^{-1}([0,\varepsilon])\) is compact for some \(\varepsilon>0\).
We give \(\mathcal A_v^{\mathrm{trap}}\) the relative topology inherited
from \(\mathcal A_v\).

\begin{theorem}\label{thm:main-spectral}
There exists an \(\iota\)-invariant dense \(G_\delta\) subset
\(\mathcal G_v^{\mathrm{trap}}\subset\mathcal A_v^{\mathrm{trap}}\) such
that, for every \(V\in\mathcal G_v^{\mathrm{trap}}\) and
\(W\in\mathcal A_v^{\mathrm{trap}}\), if
\(
 \operatorname{Spec}_{\mathrm{sc}}(P_\hbar^V)
 =\operatorname{Spec}_{\mathrm{sc}}(P_\hbar^W),
\)
then \(W=V\) or \(W=\iota V\).  Moreover, the asymptotic semiclassical spectrum determines the potential within the class of $\mathcal A_v^{\mathrm{trap}}$ up to spatial inversion generically. The same conclusions hold if,
for some sufficiently small fixed \(\delta>0\),
\(
 \operatorname{Spec}_{\delta}(P_\hbar^V)
 =\operatorname{Spec}_{\delta}(P_\hbar^W)
\)
for all sufficiently small \(\hbar>0\).
\end{theorem}

\subsection*{Idea of the proof and organization}

Let \(\cW\) be the algebra of formal power series in
\((x,\xi,\hbar)\), graded by 
\(\deg x_j=\deg\xi_j=1\) and \(\deg\hbar=2\).

For a formal symbol \(F\in \cW\), we use the following notation:
\begin{itemize}

\item \(F_m\) denotes its component of Weyl degree
\(m\), 
\medskip

\item  \(F^{[2r]}\) denotes the coefficient of \(\hbar^{2r}\) in $F$.

\end{itemize}

Section~\ref{sec:formal} places the inverse problem in a filtered
algebraic framework.  The basic observation is that the QBNF map, although
nonlinear, is triangular with respect to the Weyl grading: at each new degree
its dependence on the new Taylor coefficients is linear, while the nonlinear
terms involve only data from lower degrees.  We formulate this structure by collecting each pair of homological equations normalizing
\((V_{2N-1},V_{2N})\) into a finite dimensional operator
\(\mathcal A_N\).  The original infinite dimensional inverse
problem is thus converted into a sequence of finite dimensional observability
problems.  

Section~\ref{sec:two-layer} determines exactly what is
visible from \((B^{[0]},B^{[2]})\).  The central structural fact is that,
for every \(N\ge3\), the block operator \(\mathcal A_N\), which maps the new
pair \((V_{2N-1},V_{2N})\) to
\((B_{2N}^{[2]},B_{2N}^{[0]})\) after subtraction of the lower-degree
contributions, has a one-dimensional kernel.  Equivalently, fixing the edge
coefficient \(a_{1,2N-2}\) makes the reduced block invertible.  This exact
one-dimensional defect is the structural source of
the higher-order argument.

Section~\ref{sec:edge} shows that the edge coefficients $a_{1,2N-2}$ invisible to
\((B^{[0]},B^{[2]})\) become visible by \(B^{[4]}\).  More precisely, the
shifted coefficient \(B_{2N+2}^{[4]}\) couples to the one-dimensional kernel
of \(\mathcal A_N\).  The nonvanishing of this coupling is a
transversality condition and \(B^{[4]}\)  turns the missing direction at every degree into a
scalar observable.  We prove that this observability holds generically and,
after fixing \((a_{12},a_{14})\), use it to determine all subsequent edge
coefficients.  In this way the infinitely many one-dimensional defects of
\((B^{[0]},B^{[2]})\) collapse to a two-dimensional low-order ambiguity.

Section~\ref{sec:recover-a14} resolves the initial ambiguity \(a_{14}\).  At the bottom of the filtration, the
relevant information is nonlinear rather than a single nonzero coupling:
\(B_8^{[4]}\) reduces the possible values of \(a_{14}\) to at most two.  The compatibility information contained in
\(B_{10}^{[4]}\) generically separates them.   Combined
with Section~\ref{sec:edge}, this yields generic uniqueness on each
fixed-\(a_{12}\) fiber.

Section~\ref{sec:initial-pair} passes from fiberwise determination to generic
injectivity on the full coefficient space. The remaining obstruction is no longer homological: it is the
possibility that distinct low-order jets with different initial cubic pairs
\((a_{30},a_{12})\) lie in the same QBNF fiber.  We recast this as a
finite-jet separation problem.  A
finite dimensional incidence argument shows that false
coincidences occur only on a nowhere dense exceptional set. The finite-jet equations are encoded as a finite real
polynomial system whose
sign-constrained projection onto the remaining cubic parameter is a
singleton $a_{12}$, after which the
complete Taylor series is determined by recursive formulas.

The appendices describe the topology on global real analytic potentials and
collect the uniqueness and  invariance of QBNF and some Moyal-product identities.

\section{Formal normal form and homological equations}\label{sec:formal}

We work on the space of formal series in $(x,\xi,\hbar)$.  Give the variables the Weyl
degrees
\[
 \deg x_j=\deg\xi_j=1,\qquad \deg\hbar=2.
\]
For each $m\geq0$, let $\cW_m$ be the space of homogeneous
polynomials  of Weyl degree $m$, and let $\cW$ be the space of
formal sums
\[
 F=\sum_{m\geq0}F_m,\qquad F_m\in\cW_m.
\]

We introduce the Poisson bracket convention:
\begin{equation}\label{eq:poisson}
        \{A,B\}=\sum_{j=1}^2
        \left(\partial_{\xi_j}A\,\partial_{x_j}B
        -\partial_{x_j}A\,\partial_{\xi_j}B\right).
\end{equation}

\subsection{Moyal product}

For formal series in independent variables
\[
 A=A(x,\xi), B=B(y,\eta),
\]
define
\[
 \Lambda(AB)
 :=
 \sum_{j=1}^2
 \left(
 (\partial_{\xi_j}A)(\partial_{y_j}B)
 -(\partial_{x_j}A)(\partial_{\eta_j}B)
 \right).
\]
Then
\[
 \Lambda^k(AB)
 =
 \sum_{|\alpha|+|\beta|=k}
 (-1)^{|\beta|}
 \frac{k!}{\alpha!\beta!}
 (\partial_\xi^\alpha\partial_x^\beta A)
 (\partial_y^\alpha\partial_\eta^\beta B),
\]
where \(\alpha,\beta\in\mathbb N^2\) are multi-indices.

Define the iterated Poisson bidifferentials by
\begin{equation}\label{eq:moyal-series}
        \{A,B\}_q
        =
        \left.\Lambda^qA(x,\xi)B(y,\eta)\right|_{(y,\eta)=(x,\xi)}.
\end{equation}
In particular, \(\{A,B\}_1\) is the Poisson bracket
\eqref{eq:poisson}

 The Moyal product is
\begin{equation}\label{eq:moyal-product}
        A*B:=
        \left.
        \exp\left(\frac{\hbar}{2i}\Lambda\right)
        A(x,\xi)B(y,\eta)
        \right|_{(y,\eta)=(x,\xi)}
        =
        \sum_{q\ge0}\frac1{q!}\left(\frac{\hbar}{2i}\right)^q
        \{A,B\}_q.
\end{equation}
For the Weyl composition formula underlying \eqref{eq:moyal-product} and
its semiclassical expansion, see \cite[\S4.3, Theorems~4.11--4.12]{Zwo12}.

Since \(\{B,A\}_q=(-1)^q\{A,B\}_q\), only odd \(q\) contribute to the
commutator.  Hence the Moyal bracket satisfies
\begin{equation}\label{eq:moyal-bracket}
        \frac{i}{\hbar}[A,B]_*:=\frac{i}{\hbar}(A*B-B*A)
        =\sum_{q\ge0}
        \frac{(-1)^q\hbar^{2q}}{2^{2q}(2q+1)!}
        \{A,B\}_{2q+1}.
\end{equation}
For the two higher coefficients used below, set
\[
 \gamma_3:=\frac1{24}, \qquad \gamma_5:=\frac1{1920}.
\]
In particular,
\begin{equation}\label{eq:moyal-f}
        \frac{i}{\hbar}[A,B]_*
        =\{A,B\}-\gamma_3\hbar^2\{A,B\}_3
        +\gamma_5\hbar^4\{A,B\}_5+\cdots.
\end{equation}

Since $H_2$ is quadratic, its Moyal bracket has no higher Moyal corrections:
for every formal symbol \(A\),
\[
 \frac{i}{\hbar}[A,H_2]_* = \{A,H_2\}.
\]

With
\(z_j=x_j+i\xi_j\) and \(\bar z_j=x_j-i\xi_j\), the convention
\eqref{eq:poisson} gives
\[
 \{z_j,H_2\}=iv_jz_j,\qquad
 \{\bar z_j,H_2\}=-iv_j\bar z_j.
\]
Thus
\begin{equation}\label{eq:diagonal-L}
 \frac{i}{\hbar}[\hbar^{2r}z^\alpha\bar z^\beta,H_2]_*
 =i\langle v,\alpha-\beta\rangle
 \hbar^{2r}z^\alpha\bar z^\beta,
\end{equation}
where \(v=(v_1,v_2)\).  Since \(v_1/v_2\notin\mathbb Q\), a monomial
\(\hbar^{2r}z^\alpha\bar z^\beta\) commutes with \(H_2\) exactly when
\(\alpha=\beta\); we call such monomials resonant.  Set
\begin{equation}\label{eq:L-def}
 L(A):=\{A,H_2\}.
\end{equation}

The initial Schr\"odinger symbol \(H\) is independent of \(\hbar\), while the
Moyal bracket expansion \(\frac{i}{\hbar}[\cdot,\cdot]_*\) involves only even
powers of \(\hbar\).  The normal form construction therefore remains in the
even-\(\hbar\) subspace of \(\cW\).  We henceforth restrict \(L\) to this
subspace and let \(\Pi\) denote the projection onto the resonant monomials.
By \eqref{eq:diagonal-L}, each monomial
\(\hbar^{2r}z^\alpha\bar z^\beta\) is an eigenvector of \(L\) with eigenvalue
\(i\langle v,\alpha-\beta\rangle\).  Since \(v_1/v_2\notin\mathbb Q\),
it follows that
\[
 \ker L=\operatorname{range}\Pi
 =\operatorname{span}\{\hbar^{2r}\Omega_1^k\Omega_2^\ell:r,k,\ell\geq0\}.
\]
Thus \(L\) is invertible on the nonresonant subspace \(\ker\Pi\); we denote
this inverse by \(L^{-1}\), with the convention that $L^{-1}F$ is nonresonant.  Consequently, every homogeneous \(F\) admits a unique decomposition of the form
\begin{equation}\label{eq:splitting}
        F=\Pi F+L(G),\qquad \Pi G=0.
\end{equation}

\subsection{The quantum Birkhoff normal form}
For \(S\in\cW\), define the Moyal derivation
\[
 D_S(F):=\frac{i}{\hbar}[S,F]_*.
\]
If \(S\in\cW_m\) and \(F\in\cW_\ell\), then
\[
 D_S(F)\in\cW_{m+\ell-2}.
\]

The following theorem is the nonresonant, gauge-fixed specialization of the
formal quantum Birkhoff normal form construction; compare
\cite[Theorem~2.1 and Corollary~2.8]{CVN08}, and see also
\cite[Theorem~11.1]{VuN06}.  
Here we use a single accumulated generator throughout.

\begin{theorem}\label{thm:qbnf-construction}
Let
\[
 H=H_2+V_3+V_4+\cdots, \quad V_m\in\cW_m,
\]
be as in \eqref{eq:H-main}.  If \(v_1/v_2\notin\mathbb Q\), then
there is a unique resonant formal series
\[
 \mathcal B(H)=H_2+B_4+B_6+\cdots,\qquad B_{2N}\in\operatorname{range}\Pi,
\]
obtained from \(H\) by a formal Moyal automorphism: there is
a formal generator 
\[S=S_3+S_4+\cdots,\] with the fixed gauge \(\Pi S_m=0\) such
that
\[
 \exp(D_S)H=\mathcal B(H).
\]
\end{theorem}

The generator is constructed inductively.  Put \(S^{(1)}=0\), and suppose
\[
 S^{(N-1)}=S_3+\cdots+S_{2N-2}
\]
has already been chosen and that
\[
        \exp(D_{S^{(N-1)}})H
        =H_2+B_4+\cdots+B_{2N-2}
        +F_{2N-1}+F_{2N}+O(2N+1).
\]
Here and below \(O(m)\) denotes terms of Weyl degree at least \(m\).
 Define
\[
 S^{(N)}=S^{(N-1)}+S_{2N-1}+S_{2N}.
\]
The degree-\((2N-1)\) homological equation is
\begin{equation}\label{eq:single-formal-odd}
 [\exp(D_{S^{(N)}})H]_{2N-1}
 =F_{2N-1}+L(S_{2N-1})=0.
\end{equation}
Set 
\[
 S_{2N-1}=-L^{-1}(F_{2N-1}),
\]
and
\begin{equation*}
        G_{2N}=
        \left[\exp\!\left(D_{S^{(N-1)}+S_{2N-1}}\right)H\right]_{2N}.
\end{equation*}
Then the degree-\(2N\) homological equation is
\begin{equation}\label{eq:single-formal-degree-even}
       [\exp(D_{S^{(N)}})H]_{2N}= G_{2N}+L(S_{2N})=B_{2N}.
\end{equation}
Set
\begin{equation}\label{eq:single-formal-S}                   
        S_{2N}=-L^{-1}(G_{2N}-B_{2N}).          
\end{equation}
With these choices the induction advances to the form
\begin{equation}\label{eq:single-formal-next}
        \exp(D_{S^{(N)}})H
        =
        H_2+B_4+\cdots+B_{2N}
        +F_{2N+1}+F_{2N+2}+O(2N+3).
\end{equation}

\begin{proof}[Proof of Theorem~\ref{thm:qbnf-construction}]
For \(m\geq3\), put \(H_m:=V_m\), so that
\(H=\sum_{m\geq2}H_m\).  Through degree \(2N\),
\begin{equation}\label{eq:exponential}
\begin{aligned}
        \exp(D_{S^{(N)}})H
        &=
        \sum_{r\ge0}\frac1{r!}
        \sum_{3\le j_1,\ldots,j_r\le2N}
        \sum_{q\ge2}
        D_{S_{j_1}}\cdots D_{S_{j_r}}H_q,\\
        \deg\!\left(D_{S_{j_1}}\cdots D_{S_{j_r}}H_q\right)
        &=q+\sum_{\nu=1}^r(j_\nu-2).
\end{aligned}
\end{equation}
Terms
with no occurrence of \(S_{2N-1}\) or \(S_{2N}\) are precisely the
corresponding terms of \(\exp(D_{S^{(N-1)}})H\).

Comparing degrees in \eqref{eq:exponential}, \(S_{2N-1}\) enters degree
\(2N-1\) only through \(L(S_{2N-1})\), whereas \(S_{2N}\) enters degree
\(2N\) only through \(L(S_{2N})\); since no odd resonant monomials exist,
the splitting \eqref{eq:splitting} gives
\eqref{eq:single-formal-odd}--\eqref{eq:single-formal-S} and determines
\(S_{2N-1},S_{2N}\) uniquely under the gauge
\(\Pi S_{2N-1}=\Pi S_{2N}=0\).  Collecting
the remaining degree \(2N+1\) and \(2N+2\) components into
\(F_{2N+1}\) and \(F_{2N+2}\) gives \eqref{eq:single-formal-next}.  Iterating over
\(N\ge2\) constructs the accumulated generator and the resonant normal form.

The uniqueness assertion is proved in
Appendix~\ref{app:moyal-identities}.

\end{proof}

\begin{lemma}\label{lem:finite-step-polynomial}
Fix the nonresonant frequencies \(v_1,v_2\).  For \(m\geq3\) and
\(k\geq2\),
\(S_m\) depends polynomially on the jet
\(j^mV\), and \(B_{2k}\) depends polynomially on
\(j^{2k}V\). Moreover, the coefficient of \(\hbar^{2r}\) in every homogeneous expression of Weyl degree \(d\) arising in the normalization depends only on \(j^{d-r}V\). In particular, $B_{2m}^{[2r]}$ depends on $j^{2m-r}V$.
\end{lemma}

\begin{proof}
In the construction of
Theorem~\ref{thm:qbnf-construction}, put
\(
 G_m=\left[\exp(D_{S_3+\cdots+S_{m-1}})H\right]_m.
\)
The degree identity in \eqref{eq:exponential} shows that \(G_m\) is a finite
sum of iterated Moyal brackets involving only the fixed \(H_2\), the terms
\(V_3,\ldots,V_m\), and the generators \(S_3,\ldots,S_{m-1}\).
Starting with \(G_3=V_3\), induction and the
bilinearity of the Moyal bracket therefore make \(G_m\) a polynomial in
\(V_3,\ldots,V_m\).  Finally,
\(
 B_m=\Pi G_m, \, S_m=-L^{-1}(G_m-B_m),
\)
where \(B_m=0\) for odd \(m\).  For fixed frequencies, \(\Pi\) and
\(L^{-1}\) are fixed linear maps, so the assertion follows.

 To show that the coefficient of \(\hbar^{2r}\) in every homogeneous expression of Weyl degree \(d\) arising in the normalization depends only on \(j^{d-r}V\), it suffices to check that this property is preserved by the normalization operations. Let \(A_a\) and \(F_b\) have Weyl degrees \(a,b\geq3\) and satisfy the property. The coefficient of \(\hbar^{2r}\) in \(D_{A_a}F_b\) is a sum of constant multiples of
\(\{A_a^{[2p]},F_b^{[2q]}\}_{2s+1},\) with \(p+q+s=r.\)
Such a term depends only on the jet of order \(\max\{a-p,b-q\}\). Moreover, the bracket differentiates each factor \(2s+1\) times. Thus a nonzero term requires
\(a-4p\geq2s+1,\, b-4q\geq2s+1.\)
Since \(a,b\geq3\), these inequalities imply \(p+s\leq a-2\) and \(q+s\leq b-2\). Consequently,
\(\max\{a-p,b-q\}\leq a+b-2-r.\)
As \(D_{A_a}F_b\) has Weyl degree \(a+b-2\), the required property is preserved.
 The operators \(\Pi\) and \(L^{-1}\) preserve both the Weyl degree and the $\hbar$-powers and are independent of the higher jets of \(V\). Hence the property passes through the homological recursion and holds for every normal-form coefficient. 
\end{proof}

\medskip



\subsection{Homological equations and operators}\label{sec:homological}

\medskip

\begin{lemma}\label{lem:L-inverse-monomial}
Let \(k,\ell\geq0\) and suppose that \(k+\ell\) is odd.  Set
\[
 \omega_{a,b}^{k,\ell}:=(2a-k)v_1+(2b-\ell)v_2,\qquad W_{a,b}^{k,\ell}:=z_1^a\bar z_1^{k-a}z_2^b\bar z_2^{\ell-b}.
\]
Then we have
\begin{equation*}
 L^{-1}(x_1^kx_2^\ell)
 =
 \frac{1}{2^{k+\ell-1}}
 \sum_{\substack{0\leq a\leq k,\ 0\leq b\leq\ell\\
                  \omega_{a,b}^{k,\ell}>0}}
 \binom{k}{a}\binom{\ell}{b}
 \frac{\operatorname{Im} \!
 W_{a,b}^{k,\ell}}
 {\omega_{a,b}^{k,\ell}}.
\end{equation*}
\end{lemma}

\begin{proof}
For fixed $k,\ell$, with $z_i=x_i+i\xi_i$, $i=1,2$, then
\begin{equation}\label{eq:binomial}
 x_1^kx_2^\ell
 =
 \frac{1}{2^{k+\ell}}
 \sum_{a=0}^k\sum_{b=0}^\ell
 \binom{k}{a}\binom{\ell}{b}
 W_{a,b}^{k,\ell}.
\end{equation}
Applying \eqref{eq:diagonal-L} with $\alpha=(a,b)$ and $\beta=(k-a,\ell-b)$ gives
\[
 L(W_{a,b})=i\omega_{a,b}^{k,\ell}W_{a,b}^{k,\ell}.
\]
The terms \(W_{a,b}\) and \(W_{k-a,\ell-b}\) are complex conjugates,
while
\(
 \omega_{k-a,\ell-b}^{k,\ell}=-\omega_{a,b}^{k,\ell}.
\)
After applying \(L^{-1}\), their sum is twice the imaginary part of the
first term divided by \(\omega_{a,b}^{k,\ell}\).  
\end{proof}

\begin{lemma}\label{lem:degree-three}
The degree-three homological equation is
\begin{equation}\label{eq:cubic-eq}
        V_3+L(S_3)=0,
        \qquad \Pi S_3=0.
\end{equation}
Consequently, \(S_3\) is given by the formula
\begin{align}\label{eq:S3-formula}
        S_3={}&-
        \frac{a_{30}}{v_1}
        \left(x_1^2\xi_1+\frac23\xi_1^3\right)\notag\\
        &+\frac{a_{12}}{v_1\delta}
        \left(
        (-v_1^2+2v_2^2)\xi_1x_2^2
        +2v_2^2\xi_1\xi_2^2
        +2v_1v_2x_1x_2\xi_2
        \right),
\end{align}
where $\delta=v_1^2-4v_2^2\neq0$.
\end{lemma}

\begin{proof}
The degree-three component of \(\exp(D_{S^{(2)}})H\) is
\(V_3+L(S_3)\).
  There are no resonant cubic monomials, hence \eqref{eq:cubic-eq} follows.

Since
 $V_3=a_{30}x_1^3+a_{12}x_1x_2^2$,
applying
Lemma~\ref{lem:L-inverse-monomial} with \((k,\ell)=(3,0)\) and \((k,\ell)=(1,2)\), we get
\eqref{eq:S3-formula}. 
\end{proof}

\begin{lemma}\label{lem:first-quantum}
The degree-four homological equation is
\begin{equation}\label{eq:degree-four-eq}
        B_4
        =
        V_4+L(S_4)+\frac12D_{S_3}V_3.
\end{equation}
In particular, the classical part gives
\begin{equation}\label{eq:initial-even-block}
        \Pi V_4
        =
        B_4^{[0]}-\Pi\left(\frac12\{S_3,V_3\}\right).
\end{equation}
And the scalar coefficient of $\hbar^2$ in $B_4$ is
\begin{equation}\label{eq:b100-formula}
        b_{1,0,0}=\frac{a_{30}^2}{2v_1}
        -\frac{a_{12}^2v_2^2}{2v_1\delta}.
\end{equation}
\end{lemma}

\begin{proof}
In the expansion of \(\exp(D_{S^{(2)}})H\), the  degree-four terms are \(V_4\),
\(D_{S_4}H_2\), \(D_{S_3}V_3\), and
\(\frac12D_{S_3}^2H_2\).  Since
\(D_{S_3}H_2=-V_3\), we have
\(\frac12D_{S_3}^2H_2=-\frac12D_{S_3}V_3\).  Thus
\([\exp(D_{S^{(2)}})H]_4=B_4\) gives \eqref{eq:degree-four-eq}.
Apply \(\Pi\) to \eqref{eq:degree-four-eq}, its \(\hbar^2\)-part is
\(
 b_{1,0,0} =-\frac{\gamma_3}{2}\Pi\{S_3,V_3\}_3
\), where
\[
\begin{aligned}
        \{S_3,V_3\}_3
        =
        \partial_{\xi_1}^3S_3\,\partial_{x_1}^3V_3
        +3\partial_{\xi_1}\partial_{\xi_2}^2S_3\,
        \partial_{x_1}\partial_{x_2}^2V_3
        =-\frac{24a_{30}^2}{v_1}
        +\frac{24a_{12}^2v_2^2}{v_1\delta}.
\end{aligned}
\]
\end{proof}



\begin{lemma}\label{lem:explicit-known-remainders}
For \(N\ge3\), we have
\begin{align}\label{eq:known-remainders}
 \exp(D_{S^{(N-1)}})H
 &=
 H_2+B_4+\cdots+B_{2N-2}
 +(V_{2N-1}+K_{2N-1}) \notag\\
 &\quad
 +(V_{2N}+D_{S_3}V_{2N-1}+K_{2N})
 +O(2N+1),
\end{align}
where the remainders \(K_{2N-1}\) and \(K_{2N}\) depend on the lower-degree 
jet \(j^{2N-2}V\) .
\end{lemma}

\begin{proof}
 For \(m\geq3\), put
 \(H_m:=V_m\).  In the exponential expansion \eqref{eq:exponential}, seperating out the terms with $(V_{2N-1},V_{2N})$ and putting the remaining terms into \(K_{2N-1}\) and \(K_{2N}\), we get \eqref{eq:known-remainders} with
\begin{align}
K_{2N-1}
={}&\sum_{r\geq1}\frac1{r!}
  \sum_{\substack{3\leq j_1,\ldots,j_r\leq2N-2,\;2\leq q\leq2N-2\\
  q+\sum_{\nu=1}^r(j_\nu-2)=2N-1}}
  D_{S_{j_1}}\cdots D_{S_{j_r}}H_q,
                                                        \label{eq:K-odd-explicit}\\
K_{2N}
={}&\sum_{r\geq1}\frac1{r!}
  \sum_{\substack{3\leq j_1,\ldots,j_r\leq2N-2,\;2\leq q\leq2N-2\\
  q+\sum_{\nu=1}^r(j_\nu-2)=2N}}
  D_{S_{j_1}}\cdots D_{S_{j_r}}H_q.
                                                        \label{eq:K-even-explicit}
\end{align}

By
Lemma~\ref{lem:finite-step-polynomial}, $S_m$ depends polynomially on the jet \(j^mV\), hence \(K_{2N-1}\) and \(K_{2N}\) depend polynomially on the jet \(j^{2N-2}V\).
\end{proof}

\begin{lemma}\label{lem:odd-homological}
The degree-\((2N-1)\) equation is
\[
 K_{2N-1}+V_{2N-1}+\{S_{2N-1},H_2\}=0.
\]
Consequently, the gauge \(\Pi S_{2N-1}=0\) gives
\begin{equation}\label{eq:S-odd-solution}
 S_{2N-1}=-L^{-1}(V_{2N-1}+K_{2N-1}).
\end{equation}
\end{lemma}

\begin{proof}
This is \eqref{eq:single-formal-odd} with
\(F_{2N-1}=V_{2N-1}+K_{2N-1}\).
\end{proof}

\begin{lemma}\label{lem:degree-even-equation}
Let \(S_{2N-1}\) be chosen by \eqref{eq:S-odd-solution}, and 
\(
        G_{2N}=
        \left[
        \exp\!\left(D_{S^{(N-1)}+S_{2N-1}}\right)H
        \right]_{2N}.
\)
Then
\begin{align}\label{eq:degree-even-general}
        G_{2N}
        ={}&K_{2N}+V_{2N}
        +D_{S_3}V_{2N-1}
        +D_{S_{2N-1}}V_3 \notag\\
        &+\frac12D_{S_3}D_{S_{2N-1}}H_2
        +\frac12D_{S_{2N-1}}D_{S_3}H_2.
\end{align}
\end{lemma}

\begin{proof}
In the expansion defining \(G_{2N}\), the
terms already present in \eqref{eq:known-remainders} contribute
\[
        K_{2N}+V_{2N}+D_{S_3}V_{2N-1}.
\]
The only additional degree \(2N\) terms involving the newly inserted
\(S_{2N-1}\) are
\[
        D_{S_{2N-1}}V_3,\qquad
        \frac12D_{S_3}D_{S_{2N-1}}H_2,\qquad
        \frac12D_{S_{2N-1}}D_{S_3}H_2.
\]
Thus we conclude \eqref{eq:degree-even-general}.  
\end{proof}
We now describe the associated-graded operators used in the inverse argument. 

{\bf Notation:} We denote the basis of homogeneous space by
\[
 e_{N,m}=x_1^{2N-1-2m}x_2^{2m},
 \qquad f_{N,m}=x_1^{2N-2m}x_2^{2m}
\]
and set odd and even degree space by
\[
 \mathsf O_{2N-1}
 =\spanop\{e_{N,m}:0\leq m\leq N-1\},\qquad \mathsf E_{2N}=\spanop\{f_{N,m}:0\le m\le N\}
\]
For any $q\geq 0$, denote the resonant basis by
\[
        r^{[q]}_{N,j}=\Omega_1^{N-q-j}\Omega_2^j, \qquad 0\le j\le N-q,
\]
and set
\[
 \mathsf N_{2N}^{[q]}
 =\spanop\{r^{[q]}_{N,j}:0\le j\le N-q\}.
\]
For a formal series \(F\), we use
\(
 [r_{N,j}^{[q]}]F
\)
to denote the coefficient of \(r_{N,j}^{[q]}\) in the resonant expansion of \(F\).

The bases
are ordered by increasing index when matrices are written. Also, we denote by 
\[V^{<N}=V_3+\cdots+V_{2N-2}\]
the truncated series of the potential \(V\).


\begin{proposition}\label{prop:homological-block}
For \(N\ge3\), consider block operator
\[
 \cA_N:=
 \begin{pmatrix}
  \cQ_N&0\\
  \cC_N&\cE_N
 \end{pmatrix}:
 \mathsf O_{2N-1}\oplus\mathsf E_{2N}
 \longrightarrow
 \mathsf N_{2N}^{[2]}\oplus\mathsf N_{2N}^{[0]},
\]
where the component maps are defined by
\begin{alignat}{2}
 \cQ_N:\mathsf O_{2N-1}
 &\longrightarrow\mathsf N_{2N}^{[2]},\quad&
 \cQ_N(F)&=-\gamma_3\Pi\{S_3,F\}_3,
 \label{eq:Q-def}\\
 \cC_N:\mathsf O_{2N-1}
 &\longrightarrow\mathsf N_{2N}^{[0]},\quad&
 \cC_N(F)&=\Pi\{S_3,F\},
 \notag\\
 \cE_N:\mathsf E_{2N}
 &\longrightarrow\mathsf N_{2N}^{[0]},\quad&
 \cE_N(F)&=\Pi F.
 \label{eq:E-def}
\end{alignat}
Then the degree $2N-1$ and $2N$ homological equations \eqref{eq:single-formal-odd}--\eqref{eq:single-formal-degree-even} are equivalent to 
\begin{equation}\label{eq:block-matrix}
 \cA_N
 \begin{pmatrix}
  V_{2N-1}\\
  V_{2N}
 \end{pmatrix}
 =
 \begin{pmatrix}
  B_{2N}^{[2]}(V)-B_{2N}^{[2]}(V^{<N})\\[2mm]
  B_{2N}^{[0]}(V)-B_{2N}^{[0]}(V^{<N})
 \end{pmatrix}.
\end{equation}
Here for $r=0,1$, \(
 \widehat K_{2N}^{[2r]}:=\Pi K_{2N}^{[2r]},
\) with \(K_{2N}\) as in Lemma~\ref{lem:explicit-known-remainders}. In particular, $\widehat K_{2N}^{[2r]}$ depend only on the lower degree jet \(j^{2N-2}V\).



\end{proposition}

\begin{proof} Set $
 F:=V_{2N-1},\,\, T:=-L^{-1}F,
 \,\, U:=-L^{-1}K_{2N-1}.
$
Then \(S_{2N-1}=T+U\).  The contribution of \(U\) in $S_{2N-1}$ to $G_{2N}$ in
\eqref{eq:degree-even-general} is
\[
 D_UV_3+\frac12D_{S_3}D_UH_2+\frac12D_UD_{S_3}H_2
 =\frac12\bigl(D_UV_3-D_{S_3}K_{2N-1}\bigr).
\] where we have used \(D_UH_2=-K_{2N-1}\) and \(D_{S_3}H_2=-V_3\).
Applying \([D_U,D_{S_3}]=D_{D_US_3}\) to \(H_2\), one has
\[
 D_UL(S_3)-D_{S_3}L(U)=L(D_US_3).
\]
Since \(V_3=-L(S_3)\) and \(K_{2N-1}=-L(U)\), it follows that
\[
 D_UV_3-D_{S_3}K_{2N-1}=-L(D_US_3).
\]
 Hence, in \(\Pi G_{2N}\), the only lower degree jet contribution
is \(\widehat K_{2N}^{[2r]}=\Pi K_{2N}^{[2r]}\).  The remaining part depends on $(V_{2N-1}, V_{2N})$, given by
\[
 V_{2N}+D_{S_3}F+D_TV_3
 +\frac12D_{S_3}D_TH_2+\frac12D_TD_{S_3}H_2
=V_{2N}+\frac12D_{S_3}F+\frac12D_TV_3.\]
since \(D_TH_2=-F\) and \(D_{S_3}H_2=-V_3\).
Apply Lemma~\ref{lem:projected-bracket-identity}
,  we get
\[
\begin{aligned}
 (\Pi G_{2N})^{[0]}-\widehat K_{2N}^{[0]}
 &=\Pi V_{2N}
   +\frac12\Pi\{S_3,F\}
   +\frac12\Pi\{T,V_3\}\\
 &=\Pi V_{2N}+\Pi\{S_3,F\}\\
 &=\mathcal E_N(V_{2N})+\mathcal C_N(F).
\end{aligned}
\]
\[
\begin{aligned}
 (\Pi G_{2N})^{[2]}-\widehat K_{2N}^{[2]}
 &=-\frac{\gamma_3}{2}\Pi\{S_3,F\}_3
   -\frac{\gamma_3}{2}\Pi\{T,V_3\}_3\\
 &=-\gamma_3\Pi\{S_3,F\}_3
 =\mathcal Q_N(F).
\end{aligned}
\]
Using \(B_{2N}=\Pi G_{2N}\), we get
\begin{equation}\label{eq:block-matrix-2}
 \cA_N
 \begin{pmatrix}
  V_{2N-1}\\
  V_{2N}
 \end{pmatrix}
 =
 \begin{pmatrix}
  B_{2N}^{[2]}-\widehat K_{2N}^{[2]}\\[2mm]
  B_{2N}^{[0]}-\widehat K_{2N}^{[0]}
 \end{pmatrix}.
\end{equation}
Finally,  applying \eqref{eq:block-matrix} to the truncated potential \(V^{<N}\) gives $B_{2N}^{[r]}(V^{<N})=\widehat K_{2N}^{[r]}$ for \(r=0,2\).
\end{proof}

\section{The affine defect left by the first two QBNF layers}\label{sec:two-layer}

In this section, we determine exactly what can and cannot be recovered from the first two QBNF layers. The main contribution is the structural result in Theorem~\ref{thm:block-affine-line}, which shows that, for every \(N\ge3\), the block operator \(\cA_N\) has a one-dimensional kernel
\[
\ker\cA_N=\operatorname{span}\{(\Phi_N,\Psi_N)\}.
\]
More precisely, at each degree, the set of compatible pairs \((V_{2N-1},V_{2N})\) is an affine line whose base point $(\widehat{V_{2N-1}},\widehat{V_{2N}})$ and direction \((\Phi_N,\Psi_N)\) are determined, leaving only the edge coefficient \(a_{1,2N-2}\) as a free parameter. Thus the failure of inversion from the first two layers is exactly a one-dimensional defect at each step. This affine defect structure provides also the starting point for the \(\hbar^4\)-layer analysis in the following sections.

\medskip

For \(N\ge3\), the domain of \(\cA_N\) has dimension one larger than its
codomain.  We single out the last odd basis vector by setting
\[
 \widetilde{\mathsf O}_{2N-1}
 :=\spanop\{e_{N,m}:0\leq m\leq N-2\},
 \qquad
 \mathsf O_{2N-1}
 =\widetilde{\mathsf O}_{2N-1}
  \oplus\operatorname{span}\{e_{N,N-1}\}.
\]
Denote
\[
 V_{2N-1}^{\rm ed}:=a_{1,2N-2}e_{N,N-1},
 \qquad
 V_{2N-1}^{\rm red}:=V_{2N-1}-V_{2N-1}^{\rm ed}
 \in\widetilde{\mathsf O}_{2N-1}.
\]
We refer to \(a_{1,2N-2}\) as the edge coefficient.
Write
\[
 \cQt_N:=\cQ_N|_{\widetilde{\mathsf O}_{2N-1}},
 \qquad
 \widetilde{\cC}_N:=\cC_N|_{\widetilde{\mathsf O}_{2N-1}}.
\]

Introduce
\begin{equation}\label{eq:cn-def}
 c_0=1,
 \qquad
 c_j=2^{-2j}\binom{2j}{j},\quad j\geq1.
\end{equation}

\begin{proposition}\label{prop:even-block}
For $N\ge2$, the matrix representation of the operator
\(\cE_N:\mathsf E_{2N}\to\mathsf N_{2N}^{[0]}\) is
\begin{equation*}
        [\cE_N]=
        \operatorname{diag}(\epsilon_{N,0},\ldots,\epsilon_{N,N}), \quad \text{with}\quad \epsilon_{N,m}= c_{N-m}c_m.
\end{equation*}
In particular, \(\cE_N\) is an isomorphism.
\end{proposition}

\begin{proof}
Apply the binomial expansion \eqref{eq:binomial} with
\(k=2N-2m\) and \(\ell=2m\).  For
$f_{N,m}=x_1^{2N-2m}x_2^{2m}$, one gets
\begin{equation}\label{eq:Pi-fNm}
\begin{aligned}
        \Pi f_{N,m}
        &=2^{-2N}
        \binom{2N-2m}{N-m}
        \binom{2m}{m}
        z_1^{N-m}\bar z_1^{N-m}z_2^m\bar z_2^m\\
        &=
        c_{N-m}c_m
        \Omega_1^{N-m}\Omega_2^m.
\end{aligned}
\end{equation}
Therefore
\[
 \cE_N(f_{N,m})=\Pi f_{N,m} =\epsilon_{N,m}r_{N,m}^{[0]}.
\]
The matrix is diagonal with nonzero diagonal entries, hence \(\cE_N\) is
an isomorphism.
\end{proof}

\begin{proposition}\label{prop:odd-block}
For \(N\ge3\), the matrix representation of the operator
$\mathcal Q_N: \mathsf O_{2N-1} \to \mathsf N_{2N}^{[2]}$ is 
\begin{equation}\label{eq:Q-matrix}
 [\cQ_N]=
 \begin{pmatrix}
  \kappa_{N,0}&\lambda_{N,1}&0&\cdots&0\\
  0&\kappa_{N,1}&\lambda_{N,2}&\ddots&\vdots\\
  \vdots&\ddots&\ddots&\ddots&0\\
  0&\cdots&0&\kappa_{N,N-2}&\lambda_{N,N-1}
 \end{pmatrix}.
\end{equation}
with
\begin{equation}\label{eq:kappa}
 \kappa_{N,m}=[r^{[2]}_{N,m}]\cQ_N(e_{N,m})=\frac{a_{30}}{6v_1}
 (2n-1)(2n-2)(2n-3)c_{n-2}c_m,
\end{equation}
\begin{equation}\label{eq:lambda}
 \lambda_{N,m}= [r_{N,m-1}^{[2]}]\cQ_N(e_{N,m})=-\frac{a_{12}v_2^2}{2v_1\delta}
 (2n-1)(2m)(2m-1)c_{n-1}c_{m-1},
\end{equation}
where \(n=N-m\). Consequently, if \(a_{30}\ne0\), 
\(
 \cQt_N:\widetilde{\mathsf O}_{2N-1}\longrightarrow
 \mathsf N_{2N}^{[2]}
\) is an isomorphism.
\end{proposition}

\begin{proof} Write \(e_m=e_{N,m}\).
By definition \eqref{eq:Q-def},
\(
 \cQ_N(e_m)=-\gamma_3\Pi\{S_3,e_m\}_3.
\) Using the formula for $S_3$ in \eqref{eq:S3-formula}, a direct computation gives
\begin{equation}\label{eq:qformula}
\cQ_N(e_m)
 =
 \frac{a_{30}}{6v_1}\Pi\partial_{x_1}^3e_m
 -\frac{a_{12}v_2^2}{2v_1\delta}
   \Pi\partial_{x_1}\partial_{x_2}^2e_m.
\end{equation}
Notice that the two terms in \eqref{eq:qformula} are multiples of \(
 \Omega_1^{n-2}\Omega_2^m=r^{[2]}_{N,m}
\) and \(
 \Omega_1^{n-1}\Omega_2^{m-1}=r^{[2]}_{N,m-1}
\) respectively. We get the coefficients $\kappa_{N,m}$ and $\lambda_{N,m}$ by applying \eqref{eq:Pi-fNm}.

\end{proof}

\begin{proposition}\label{prop:block}
Assume \(a_{30}\ne0\). For \(N\ge3\), the block operator
\[
 \widetilde\cA_N=
 \begin{pmatrix}
  \cQt_N&0\\
  \widetilde{\cC}_N&\cE_N
 \end{pmatrix}:
 \widetilde{\mathsf O}_{2N-1}\oplus\mathsf E_{2N}
 \longrightarrow
 \mathsf N_{2N}^{[2]}\oplus\mathsf N_{2N}^{[0]}
\]
is an isomorphism.
\end{proposition}

\begin{proof}
It follows directly from the invertiblity of \(\cQt_N\) and \(\cE_N\) by Propositions~\ref{prop:odd-block} and \ref{prop:even-block}.
\end{proof}

\begin{theorem}\label{thm:block-affine-line}
Assume \(a_{30}\ne0\).  For \(N\geq3\),  we have the following affine line  in \(\mathsf O_{2N-1}\oplus\mathsf E_{2N}\):
\begin{equation}\label{eq:affine-line}
 (V_{2N-1},V_{2N})
 =
 (\widehat{V_{2N-1}},\widehat{V_{2N}})
 +a_{1,2N-2}(\Phi_N,\Psi_N).
\end{equation}
where the direction \((\Phi_N,\Psi_N)\) is the unique normalized generator of
the kernel of the block operator \(\cA_N\):
\begin{equation}\label{eq:block-kernel}
 \ker\cA_N
 =\operatorname{span}\{(\Phi_N,\Psi_N)\},
 \qquad
 \Phi_N-e_{N,N-1}\in\widetilde{\mathsf O}_{2N-1},\quad \Psi_N\in\mathsf E_{2N}.
\end{equation}
The base point
\((\widehat{V_{2N-1}},\widehat{V_{2N}})\) is determined by the
$(2N-2)$-jet \(j^{2N-2}V\) together with \(B_{2N}^{[0]}(V)\) and \(B_{2N}^{[2]}(V)\).
Moreover, \((\Phi_N,\Psi_N)\) depends only on
\(N,v_1,v_2,a_{30}\), and \(a_{12}\).
\end{theorem}

\begin{proof}
Put
\(
 e=e_{N,N-1},\, t=a_{1,2N-2},
\)
and denote
\[
\Delta B_{2N}^{[2]}=B_{2N}^{[2]}(V)-B_{2N}^{[2]}(V^{<N}),
 \qquad
\Delta B_{2N}^{[0]}=B_{2N}^{[0]}(V)-B_{2N}^{[0]}(V^{<N}).
\]
The block equation \eqref{eq:block-matrix} writes as
\[
 \widetilde\cA_N
 \begin{pmatrix}
  V_{2N-1}^{\rm red}\\
  V_{2N}
 \end{pmatrix}
 =
 \begin{pmatrix}
  \Delta B_{2N}^{[2]}\\ \Delta B_{2N}^{[0]}
 \end{pmatrix}
 -t
 \begin{pmatrix}
  \cQ_Ne\\
  \cC_Ne
 \end{pmatrix}.
\]
By Proposition~\ref{prop:block}, \(\widetilde\cA_N\) is invertible.  We get ~\eqref{eq:affine-line} by defining
\begin{equation}\label{eq:phi-psi}
 \begin{pmatrix}
  \widehat{V_{2N-1}}\\
  \widehat{V_{2N}}
 \end{pmatrix}
 :=
 \widetilde\cA_N^{-1}
 \begin{pmatrix}
  \Delta B_{2N}^{[2]}\\ \Delta B_{2N}^{[0]}
 \end{pmatrix},
\qquad
 \begin{pmatrix}
  \Phi_N-e\\
  \Psi_N
 \end{pmatrix}
 :=
 -\widetilde\cA_N^{-1}
 \begin{pmatrix}
  \cQ_Ne\\
  \cC_Ne
 \end{pmatrix},
\end{equation}

By the definition
of \((\Phi_N-e,\Psi_N)\), we have
\[
 \cA_N
 \begin{pmatrix}
  \Phi_N\\ \Psi_N
 \end{pmatrix}
 =0.
\]
We get \eqref{eq:block-kernel} since $\dim\ker\cA_N=1$.
The vector \((Y_N,Z_N)\), and hence the base point $(\widehat{V_{2N-1}},\widehat{V_{2N}})$, is determined by lower degree jet $j^{2N-2}V$ and $B_{2N}^{[2]}$, $B_{2N}^{[0]}$.  The normalized kernel vector \((\Phi_N,\Psi_N)\) is
determined only by \(\cA_N\), which depends on
\(N,v_1,v_2\) and \(S_3\).  We conclude the dependence by formula~\eqref{eq:S3-formula} for $S_3$.
\end{proof}

\begin{theorem}\label{thm:two-layer}
Assume  \(a_{30}\ne0\).  The first two QBNF layers
\(B^{[0]}(V)\) and \(B^{[2]}(V)\), the sign of \(a_{30}\), and  the
coefficients
\(
 \{a_{1,2k}\}_{k\geq1},
\)
determine the complete Taylor series of \(V\).
\end{theorem}

\begin{proof}
By Lemmas~\ref{lem:degree-three} and~\ref{lem:first-quantum} show that, after
fixing $a_{12}$ and the sign of \(a_{30}\), the cubic term \(V_3\) and the generator
\(S_3\) are known.
The classical degree-$4$ equation \eqref{eq:initial-even-block} then
determines \(V_4\) by Proposition~\ref{prop:even-block}, and the non-resonant part
determines \(S_4\).

By induction, assume that all \(V_m\) and \(S_m\) with \(m\le 2N-2\) are known. Since the edge
coefficient \(a_{1,2N-2}\) is prescribed, by
Theorem~\ref{thm:block-affine-line}, 
\[
 (V_{2N-1},V_{2N})
 =
 (\widehat{V_{2N-1}},\widehat{V_{2N}})
 +a_{1,2N-2}(\Phi_N,\Psi_N).
\]
Thus
\(V_{2N-1}\) and \(V_{2N}\) are determined.  Also,
\eqref{eq:S-odd-solution} and the non-resonant part of
\eqref{eq:degree-even-general} determine \(S_{2N-1}\) and \(S_{2N}\). Thus we conclude.
\end{proof}

The above theorem was first established by K.~Wang~\cite{Wang24, Wang26}. Wang's recursive argument proves the result by eliminating Taylor coefficients degree by degree; our structural theorem \ref{thm:block-affine-line} instead isolates the underlying one-dimensional affine defect and identifies \(a_{1,2N-2}\) as its natural parameter for every \(N\ge3\).

\begin{proposition}\label{prop:rational-affine-dependence} For every $V\in \mathcal V_{v}^{\sigma}$ and $N\geq 3$,
the coefficients of the normalized direction \((\Phi_N(V),\Psi_N(V))\) are
rational functions of \(a_{30}\) and \(a_{12}\), while the coefficients
of the base point
\((\widehat {V_{2N-1}},\widehat {V_{2N}})\) are rational functions of
\(j^{2N-2}V\), \(B_{2N}^{[2]}(V)\), and \(B_{2N}^{[0]}(V)\).  In both cases,
the only parameter-dependent denominators are powers of \(a_{30}\).
\end{proposition}

\begin{proof}
Observe that given nonresonant frequencies $(v_1,v_2)$, \(L^{-1}\) and \(\Pi\) are fixed
linear maps, and every homogeneous Moyal operation is a finite
polynomial expression on  the inputs.
Consequently, these operations preserve polynomial dependence on the
jet coordinates of $V$.  In particular, Lemmas~\ref{lem:finite-step-polynomial}
and~\ref{lem:explicit-known-remainders} show that the lower degree terms
\(\widehat K_{2N}^{[2]}\) and \(\widehat K_{2N}^{[0]}\) are polynomial
in \(j^{2N-2}V\).

Since \(S_3=-L^{-1}V_3\), its coefficients are linear in \(a_{30}\) and
\(a_{12}\).  Definitions~\eqref{eq:Q-def}--\eqref{eq:E-def} therefore
show that every entry of \(\widetilde\cA_N\), as well as every entry of
\((\cQ_Ne_{N,N-1},\cC_Ne_{N,N-1})\), is polynomial in \(a_{30}\) and
\(a_{12}\).  The matrix \(\widetilde\cA_N\) is block lower triangular.
By Proposition~\ref{prop:odd-block}, ~\ref{prop:even-block}, 
\(
 \det\widetilde\cA_N=\widetilde c_Na_{30}^{N-1},\) for some \(
 \widetilde c_N\ne0.
\)
Hence every entry of
\(\widetilde\cA_N^{-1}\) is rational in \(a_{30},a_{12}\), with a power
of \(a_{30}\) as its only parameter-dependent denominator.  

We conclude by applying the defining formula \eqref{eq:phi-psi} since \((\Delta B_{2N}^{[2]},\Delta B_{2N}^{[0]})\) is polynomial in \(j^{2N-2}V\),
\(B_{2N}^{[2]}\), and \(B_{2N}^{[0]}\).  .
\end{proof}

In view of the structure theorem~\ref{thm:block-affine-line}, we will denote the affine line family by
\begin{equation}\label{eq:affine-line-family}
 \mathcal (V_{2N-1}(t),V_{2N}(t)):=
(\widehat{V_{2N-1}},\widehat{V_{2N}})
 +t(\Phi_N,\Psi_N),\quad t\in \mathbb R.
\end{equation}
and correspondingly define the affine family of truncated Taylor series
\begin{equation}\label{eq:affine-line-family-jet}
V^{<N+1}(t):=V^{<N}+V_{2N-1}(t)+V_{2N}(t),\quad H^{<N+1}(t):=H^{<N}+V_{2N-1}(t)+V_{2N}(t).
\end{equation}

\section{Generic recovery of the higher edge coefficients}\label{sec:edge}

In this section, we reduce the infinite family of edge coefficients to finite-jet generic conditions and recover all higher edge coefficients once \(a_{12}\) and \(a_{14}\) are known.

For
\(\alpha,\beta\in\mathbb R\), set
\[
 \mathcal V_{v,\sigma}^{\alpha,\beta}
 :=\{V\in\mathcal V_v^\sigma:a_{12}(V)=\alpha,\ a_{14}(V)=\beta\},
\]
with the relative coefficientwise topology.

\begin{theorem}\label{thm:generic-edge}
Fix \(\sigma\in\{\pm1\}\), \(\alpha\neq0\), and \(\beta\).  There is a dense \(G_\delta\) subset
\(\mathcal G_{v,\sigma}^{\alpha,\beta}
\subset\mathcal V_{v,\sigma}^{\alpha,\beta}\) such that, for every
\(V\in\mathcal G_{v,\sigma}^{\alpha,\beta}\) and
\(W\in\mathcal V_{v,\sigma}^{\alpha,\beta}\),
\[
 (B^{[0]},B^{[2]},B^{[4]})(W)
 =(B^{[0]},B^{[2]},B^{[4]})(V)
 \quad\Longrightarrow\quad W=V.
\]
Moreover, these three layers and the fixed values
\(\alpha,\beta,\sigma\) determine every
\(V\in \mathcal G_{v,\sigma}^{\alpha,\beta}\) successively.
\end{theorem}

For a formal symbol $A$, denote the first three components of its Moyal
derivation by
\begin{equation}\label{eq:PMN}
 P_A(F)=\{A,F\},
 \qquad
 M_A(F)=-\gamma_3\{A,F\}_3,
 \qquad
 N_A(F)=\gamma_5\{A,F\}_5.
\end{equation}
Thus
\[
 D_A=P_A+\hbar^2M_A+\hbar^4N_A+O(\hbar^6).
\]
Given the first three components $S_3,S_4,S_5$ of the gauge fixed generator $S$ , write
\[
 P_j=P_{S_j},\quad  M_j=M_{S_j},\quad  N_j=N_{S_j}.
\]
Denote the partial sums of the potential and the corresponding
Hamiltonian by
\[
 V^{<N}:=V_2+V_3+\cdots+V_{2N-2},
 \quad
 H^{<N}:=H_{V^{<N}}=H_2+V_3+\cdots+V_{2N-2},
\]
and $V^{\geq N}:=V-V^{<N}$.

We first recall the following initial determination lemma.

\begin{lemma}\label{lem:initial-reconstruction}
Once $a_{12}$, $a_{14}$, and the sign of $a_{30}$ are fixed, the layers $B^{[0]}$ and
$B^{[2]}$ determine $V_m,S_m$ for \(m\le 6\).

\end{lemma}

\begin{proof}
This follows directly from the proof of Theorem~\ref{thm:two-layer} up to
$N=3$.
\end{proof}

\begin{lemma}\label{lem:kernel-vector}
Let \((\Phi_N,\Psi_N)\) be the normalized generator of
\(\ker\cA_N\).  Then
\begin{equation}\label{eq:phi}
 \Phi_N=\sum_{m=0}^{N-1}\phi_{N,m}e_{N,m},
 \qquad
 \phi_{N,m}=\phi_{N,m}^*\tau^{N-1-m},
 \quad 0\leq m\leq N-1,
\end{equation}
where $\tau:=a_{12}v_2^2/(a_{30}\delta)$ and each
$\phi_{N,m}^*$ is a nonzero constant depending only on $N$ and $m$, with
$\phi_{N,N-1}^*=1$.
\end{lemma}

\begin{proof} 
 Notice that the normalization condition $\Phi_N-e_{N,N-1}\in\widetilde{\mathsf O}_{2N-1}$ implies that $\phi_{N,N-1}=1$.
By Proposition~\ref{prop:odd-block}, solving 
\(\cQ_N\Phi_N=0\) gives a recurrence relation for the coefficients \(\phi_{N,j}\): 
\[
 \kappa_{N,j}\phi_{N,j}
 +\lambda_{N,j+1}\phi_{N,j+1}=0,\quad 0\leq j\leq N-2.
\]
Iterating from $\phi_{N,N-1}=1$ proves the asserted
formula for $\phi_{N,m}$. 
\end{proof}

\begin{definition}\label{def:h4-transfer}
Given \(V\in \mathcal V_v\) and \(N\geq3\), define the \(\hbar^4\)-transfer operator
\(\mathcal T_V\) on
\(\mathsf O_{2N-1}\times\mathsf E_{2N}\) by
\begin{equation}\label{eq:T-operator}
 \mathcal T_V(F,G)
 =
 [\exp(D_{S_3(V)+S_4(V)+S_5(V)})(F+G)]_{2N+2}^{[4]}.
\end{equation}
where $S_3(V),S_4(V),S_5(V)$ are the corresponding degree components in the gauge-fixed normalized operator $S(V)$ for $H_V$.
\end{definition}

\begin{lemma}\label{lem:T-explicit}
For every \(N\geq3\), the \(\hbar^4\)-transfer operator $\mathcal T_V$
depends on the five-jet $j^5 V$ and has the form
\begin{align}
 \mathcal T_V(F,G)
 ={}&N_5F+\frac12(M_3M_4+M_4M_3)F\notag\\
 &+\frac16(P_3M_3^2+M_3P_3M_3+M_3^2P_3)F
 +\frac12M_3^2G.
 \label{eq:L-operator}
\end{align}
 In particular, it is linear in
\((F,G)\): for \(c_1,c_2\in\mathbb R\) and
\((F_j,G_j)\in\mathsf O_{2N-1}\times\mathsf E_{2N}\),
\[
 \mathcal T_V(c_1F_1+c_2F_2,c_1G_1+c_2G_2)
 =c_1\mathcal T_V(F_1,G_1)+c_2\mathcal T_V(F_2,G_2).
\]
\end{lemma}

\begin{proof}
Write \(D_j=D_{S_j}\).  Since \(D_j\) raises Weyl degree by \(j-2\),
the degree-\((2N+2)\) part of \(\exp(D_{S_3+S_4+S_5})(F+G)\) is
\begin{align*}
 D_5F+\frac12(D_3D_4+D_4D_3)F+\frac16D_3^3F
+D_4G+\frac12D_3^2G.
\end{align*}
The generators \(S_3\) and \(S_4\) are \(\hbar\)-independent and have
degrees less than five.  Hence
\[
 D_3=P_3+\hbar^2M_3, \quad D_4=P_4+\hbar^2M_4,
\]
  Moreover,
\((D_5F)^{[4]}=N_5F\).
Indeed, Weyl homogeneity gives
\(S_5=S_5^{[0]}+\hbar^2S_5^{[2]}\), where \(S_5^{[2]}\) is linear.  Therefore
\[
 (D_5F)^{[4]}
 =-\gamma_3\{S_5^{[2]},F\}_3+\gamma_5\{S_5^{[0]},F\}_5
 =\gamma_5\{S_5^{[0]},F\}_5=N_5F.
\]
For the remaining terms,
\begin{align*}
 (D_3D_4F)^{[4]}&=M_3M_4F,\quad (D_4D_3F)^{[4]}=M_4M_3F\\
 (D_3^3F)^{[4]}
 &=(P_3M_3^2+M_3P_3M_3+M_3^2P_3)F,\\
 (D_3^2G)^{[4]}&=M_3^2G.
\end{align*}
Thus we obtain \eqref{eq:L-operator}.  Since every
term on its right-hand side is linear in either \(F\) or \(G\), the asserted
linearity follows.
\end{proof}

\begin{lemma}\label{lem:direct-high-degree}
For \(N\geq4\), we have
\begin{equation*}
 [\mathcal B(V)-\mathcal B(V^{<N})]_{2N+2}
 =
 \Pi[\exp(D_{S_3+S_4+S_5})V^{\geq N}]_{2N+2}.
\end{equation*}
\end{lemma}

\begin{proof} Let $\operatorname{ord} F$ denote the lowest Weyl degree of a formal series $F$.
Let \(\widetilde S\) be the full gauge-fixed generator normalizing
\(H_{V^{<N}}\) and denote
\(
 R=\exp(D_{\widetilde S})V^{\geq N}.
\)
 We have
$\operatorname{ord}R\geq d=2N-1$.  Moreover,
\(
 \exp(D_{\widetilde S})H_V=\mathcal B(V^{<N})+R.
\)
Normalize \(\mathcal B(V^{<N})+R\) by $T$, with $\operatorname{ord}T\geq d$. By the uniqueness of the QBNF,
\[
 \mathcal B(V)=\exp(D_T)(\mathcal B(V^{<N})+R).
\]
Expanding the exponential, we obtain
\begin{align*}
 \mathcal B(V)
 ={}\mathcal B(V^{<N})+R+D_T\mathcal B(V^{<N})
 +\sum_{j\geq1}\frac1{j!}D_T^jR
 +\sum_{j\geq2}\frac1{j!}D_T^j\mathcal B(V^{<N}).
\end{align*}
Since
\(\operatorname{ord}T\geq d\) and \(\operatorname{ord}R\geq d\), we get
\[
 \operatorname{ord}(D_T^jR)
 \geq d+j(d-2)\geq2d-2,\quad j\geq 1,
\]
\[
 \operatorname{ord}(D_T^j\mathcal B(V^{<N}))
 \geq2+j(d-2)\geq2d-2,\quad j\geq 2.
\]
It follows that, for every \(m<2d-2\),
\[
 [\mathcal B(V)-\mathcal B(V^{<N})]_m
 =[R+D_T\mathcal B(V^{<N})]_m.
\]
Applying $\Pi$ to both sides of the above equality, one has
\[
 [\mathcal B(V)-\mathcal B(V^{<N})]_m=\Pi[R]_m,
 \qquad m<2d-2.
\]
Here we used $\Pi D_T\mathcal B(V^{<N})=0$, which follows from
Lemma~\ref{lem:resonant-moyal}.
We conclude by observing that \(\widetilde S\) and \(S^{(N-1)}\) have the same components up to degree \(2N-2\), and that 
\(\operatorname{ord}(V^{\geq N})\geq2N-1\), so its degree-\((2N+2)\)
image under the
exponential can involve only \(S_3,S_4,S_5\).  Thus
\[
 [\exp(D_{\widetilde S})V^{\geq N}]_{2N+2}
 =
 [\exp(D_{S_3+S_4+S_5})V^{\geq N}]_{2N+2}.
\]
\end{proof}
\begin{proposition}\label{prop:shifted-equation}
For \(N\geq4\), set
\(
 R_N^{[4]}(V):=B_{2N+2}^{[4]}(V)-B_{2N+2}^{[4]}(V^{<N}).
\)
Then 
\begin{equation}\label{eq:shifted-h4}
 \Pi\mathcal T_V(V_{2N-1},V_{2N})=R_N^{[4]}(V).
\end{equation}

\end{proposition}

\begin{proof}
By Lemma~\ref{lem:direct-high-degree},
\begin{equation}\label{eq:shifted-direct-comparison}
 [\mathcal B(V)-\mathcal B(V^{<N})]_{2N+2}
 =
 \Pi[\exp(D_{S_3+S_4+S_5})V^{\geq N}]_{2N+2}.
\end{equation}
Only \(V_{2N-1},V_{2N},V_{2N+1},V_{2N+2}\) can contribute on the
right-hand side.  The terms involving the last two are
\(V_{2N+2}+D_{S_3}V_{2N+1}\),
and neither has an \(\hbar^4\)-coefficient because \(S_3\) is cubic. We conclude by applying the definition of \(\mathcal T_V\) in \eqref{eq:T-operator}. 
\end{proof}

\medskip

Notice that  $\Pi\mathcal T_V(V_{2N-1},V_{2N})$ is homogeneous of
degree $N-3$ in $\Omega_1,\Omega_2$.  In the fixed action coordinates, we
define the
pure transverse coefficient to single out  the coefficient of
\(\Omega_2^{N-3}\):
\[
 \rho_N(F):=[\Omega_2^{N-3}]\Pi F.
\]

\begin{definition}\label{def:edge-coefficient}
For \(N\geq4\) and \(V\in \mathcal V_v^\sigma\), we define the edge function by
\begin{equation}\label{eq:dN-definition}
 d_N(V)
 :=\rho_N\bigl(
 \mathcal T_V(\Phi_N(V),\Psi_N(V))
 \bigr).
\end{equation}
\end{definition}

\begin{lemma}\label{lem:edge-affine-variation}
For \(N\geq4\), consider
\[
 \mathfrak b_N(t):=\rho_N\!\left(B_{2N+2}^{[4]}(V^{<N+1}(t))-B_{2N+2}^{[4]}(V^{<N})\right).
\]
where $V^{<N+1}(t)$ is the affine family of truncated Taylor series defined in \eqref{eq:affine-line-family-jet}. Then
 \(\mathfrak b_N(t)=\mathfrak b_N(0)+t\,d_N(V).\)
\end{lemma}

\begin{proof} For fixed $N$, we write $V(t)=V^{<N+1}(t)$.  
Since \(N\geq4\), the family \(V(t)\) has a fixed five-jet.  Hence
\(\mathcal T_{V(t)}\), and therefore \(d_N(V(t))\), is independent of
\(t\). By Proposition~\ref{prop:shifted-equation},
\(
 \mathfrak b_N(t)
 =
 \rho_N\mathcal T_{V(t)}(V_{2N-1}(t),V_{2N}(t)).
\)
Using  the linearity of the
transfer operator, we obtain
\begin{equation}
 \mathfrak b_N(t)
 =
 \rho_N\mathcal T_{V(t)}
   (\widehat{V_{2N-1}},\widehat{V_{2N}})
 +t\,\rho_N\mathcal T_{V(t)}(\Phi_N,\Psi_N)
 =\mathfrak b_N(0)+t\,d_N.
\end{equation}

\end{proof}

\begin{proposition}\label{cor:scalar-edge-equation}
For $N\geq4$, the edge coefficient $a_{1,2N-2}$ satisfies
\begin{equation}\label{eq:scalar-edge}
 d_N(V)a_{1,2N-2}=z_N(V),
\end{equation}
where
 \(z_N(V)
 :=\rho_N(R_N^{[4]}(V))
 -\rho_N\bigl(
 \mathcal T_V(\widehat{V_{2N-1}},\widehat{V_{2N}})
 \bigr).\)
In particular, $z_N(V)$ depends on the jet $j^{2N-2}V$ and $\mathcal{B}^{[4]}(V)$.
\end{proposition}
\begin{proof}
Take $t=a_{1,2N-2}$ in Lemma~\ref{lem:edge-affine-variation} and use
$z_N(V)=\mathfrak b_N(a_{1,2N-2})-\mathfrak b_N(0)$.
\end{proof}


\begin{proposition}\label{prop:d-nonzero}
For every $N\geq4$ and \(V\in \mathcal V_v^\sigma\), the edge function $d_N(V)$ is a rational function
of the five-jet $j^5 V$ that is not identically zero.
Moreover, 
 \begin{align}
  \frac{\partial d_N(V)}{\partial a_{50}}
  ={}&-\frac4{v_1}c_{N-3}\phi_{N,N-3}^*\tau^2,\quad \,\,\tau=a_{12}v_2^2/(a_{30}\delta)
  \label{eq:d-derivative} \end{align}
where $c_{N-3}$ and $\phi^*_{N,N-3}$ are non zero constants given in \eqref{eq:cn-def} and \eqref{eq:phi} respectively.
Consequently, the set
 \[
  \{V\in\mathcal V_v^\sigma:a_{12}(V)\neq0,\ d_N(V)\neq0\}
 \]
 is open and dense in $\mathcal V_v^\sigma$.

\end{proposition}

\begin{proof}
By Proposition~\ref{prop:rational-affine-dependence},
\((\Phi_N,\Psi_N)\) is rational in \(a_{30},a_{12}\), with only powers
of \(a_{30}\) in its denominators.  Since \(\mathcal T_V\) is polynomial
in \(j^5V\), by
definition~\eqref{eq:dN-definition}, we see that \(d_N(V)\) is rational
in \(j^5V\) with no pole on \(a_{30}\ne0\).
For simplicity, abbreviate \(d_N=d_N(V)\) and
\(\mathcal T=\mathcal T_V\).

By Lemma~\ref{lem:explicit-known-remainders}, $K_5$ depends only on the
 $j^4V$,
differentiating the degree-$5$ homological equation
\(
 K_5+V_5+L(S_5)=0
\)
with respect to $a_{50}$ therefore gives
\(
 x_1^5+L\left(\frac{\partial S_5}{\partial a_{50}}\right)=0.
\)
Thus
\[
 \frac{\partial S_5}{\partial a_{50}}=A_0,
 \qquad
 \text{where}\quad A_0:=-L^{-1}(x_1^5)=-\frac1{v_1}
 \left(x_1^4\xi_1+\frac43x_1^2\xi_1^3+\frac8{15}\xi_1^5\right).
\]
Recall that
\(
 d_N=\rho_N\mathcal T(\Phi_N,\Psi_N)
\).
In the explicit expression \eqref{eq:L-operator} for $\mathcal T$, since $S_3,S_4,\Phi_N$, and $\Psi_N$ are independent of $a_{50}$, the only term depending
on $a_{50}$ is
\(N_5\Phi_N=\gamma_5\{S_5,\Phi_N\}_5\).  Therefore
\begin{equation*}
 \frac{\partial d_N}{\partial a_{50}}
 =\gamma_5\rho_N\{A_0,\Phi_N\}_5.
\end{equation*}
Writing
\(
 \Phi_N=\sum_{m=0}^{N-1}\phi_{N,m}e_{N,m},
\)
observe that \(A_0\) depends only on \((x_1,\xi_1)\), whereas \(e_{N,m}=x_1^{2N-2m-1}x_2^{2m}\)
is independent of \(\xi\).  Thus the only term in the fifth
bidifferential that can be nonzero is
\(
 \{A_0,e_{N,m}\}_5
 =
 (\partial_{\xi_1}^5A_0)(\partial_{x_1}^5e_{N,m}).
\) Since $\partial_{\xi_1}^5A_0$ is a nonzero constant, and the resonant projection of \(\partial_{x_1}^5e_{N,m}\) is a multiple of $\Omega_1^{N-m-3}\Omega_2^m$, we get that only $m=N-3$ contributes to the fifth-order bracket.  Hence
\begin{equation}\label{eq:d-derivative-reduced}
 \frac{\partial d_N}{\partial a_{50}}
 =\gamma_5\phi_{N,N-3}
   \rho_N\big((\partial_{\xi_1}^5A_0) (\partial_{x_1}^5(x_1^5x_2^{2N-6})\big)=-\frac4{v_1}c_{N-3}\phi^*_{N,N-3}\tau^2,
\end{equation}
where we have used
\(
 \rho_N(x_2^{2N-6}) =c_{N-3}
\) and $\phi_{N,N-3}=\phi_{N,N-3}^*\tau^2$.

Finally, let $V_{\varepsilon}=V+\varepsilon x_1^5$, then $d_N(V_{\varepsilon})=d_N(V)+\varepsilon m_N(V)$, where \(
 m_N(V):=\frac{\partial d_N}{\partial a_{50}}(V) 
\) is a nonzero constant. Thus arbitary small perturbation can make $d_N(V)$ nonzero, this proves the density. Openness follows from the fact that $d_N(V)$ is rational with no poles.

\end{proof}

\begin{proof}[Proof of Theorem~\ref{thm:generic-edge}]
Define
\[
 \mathcal O_N^{\alpha,\beta}
 :=
 \{V\in\mathcal V_{v,\sigma}^{\alpha,\beta}:d_N(V)\neq0\},\qquad 
 \mathcal G_{v,\sigma}^{\alpha,\beta}
 :=
 \bigcap_{N\geq4}\mathcal O_N^{\alpha,\beta}.
\]
The set $\mathcal O_N^{\alpha,\beta}$ is open and dense. For any $V\in\mathcal G_{v,\sigma}^{\alpha,\beta}$, the edge coefficient $d_N(V)$ is nonzero for all $N\geq4$.
The supplied coefficients $a_{12},a_{14}$ and the fixed sign of $a_{30}$, together with
$B^{[0]}$ and $B^{[2]}$, determine the six-jet $j^6V$ by
Lemma~\ref{lem:initial-reconstruction}.  Suppose inductively that
the $(2N-2)$-jet $j^{2N-2}V$, where $N\geq4$, has been determined. Since
$d_N\neq0$, that edge equation \eqref{eq:scalar-edge} determines $a_{1,2N-2}$, and then \eqref{eq:affine-line} determines
$V_{2N-1}$ and $V_{2N}$.  Thus the induction determines all of $V$.

For any $W\in\mathcal V_{v,\sigma}^{\alpha,\beta}$, if
\((B^{[0]},B^{[2]},B^{[4]})(W)=(B^{[0]},B^{[2]},B^{[4]})(V)\), then
\(j^6W=j^6V\).  Since \(d_N\) depends
only on the five-jet, $d_N(W)=d_N(V)\neq0$ for all $N\geq4$.  The same induction applied to \(W\) produces the same
sequence of coefficients as for \(V\).  Hence \(W=V\).


Consequently,
\(
 \mathcal G_{v,\sigma}^{\alpha,\beta}
\)
is a dense \(G_\delta\) subset that meets the requirements of the theorem.

\end{proof}

\section{Determination of the initial edge coefficient
\texorpdfstring{$a_{14}$}{a14}}\label{sec:recover-a14}

Throughout this section, the nonresonant frequencies
$v=(v_1,v_2)$ and \(\sigma\in\{\pm1\}\) are fixed.  Except in the statement and proof of
Theorem~\ref{thm:recover-a14}, the constructions are made on
$\mathcal V_v^\sigma$, without prescribing the value of $a_{12}$.

We prove that the first three QBNF
layers $B^{[0]},B^{[2]},B^{[4]}$ generically determine $a_{14}$ and
then the whole Taylor series.

The argument has a simple structure.  The degree-six equations leave
one affine parameter, namely $a_{14}$.  The coefficient $B^{[4]}_8$
reduces this affine line to at most two candidates.  The full
two-component action polynomial $B^{[4]}_{10}$ then gives a compatibility
equation, which generically rejects the false candidate.

For $\alpha\neq 0$, define the fixed coefficient slice
\[
\mathcal V_{v,\sigma}^{\alpha}
 :=\{V\in\mathcal V_v^\sigma:a_{12}(V)=\alpha\}.
\]
We give this fixed coefficient slice its relative cylinder topology. 

\begin{theorem}\label{thm:recover-a14}
Fix $\alpha\neq0$.  There is a dense
\(G_\delta\) subset
\(
 \mathcal G_{v,\sigma}^{\alpha}
 \subset\mathcal V_{v,\sigma}^{\alpha}
\)
such that, for every \(V\in\mathcal G_{v,\sigma}^{\alpha}\) and
\(W\in\mathcal V_{v,\sigma}^{\alpha}\),
\[
 (B^{[0]},B^{[2]},B^{[4]})(W)
 =(B^{[0]},B^{[2]},B^{[4]})(V)
 \quad\Longrightarrow\quad W=V.
\]
Moreover, these three layers 
determine \(V\in\mathcal G_{v,\sigma}^{\alpha}\) successively.
\end{theorem}

\begin{lemma}\label{lem:direct-high-degree-N3}
Given a potential $V\in \mathcal V_v^\sigma$, let \(S^\circ\) be the
gauge-fixed generator normalizing \(H_{V^{<3}}\).
Then
\begin{equation}\label{eq:direct-high-degree-N3}
 [\mathcal B(V)-\mathcal B(V^{<3})]_8
 =
 \Pi[\exp(D_{S^\circ})V^{\geq3}]_8
 -\frac12\Pi D_{L^{-1}V_5}V_5.
\end{equation}
\end{lemma}

\begin{proof}
Denote $
 R:=\exp(D_{S^\circ})V^{\geq3}.
$
As in the proof of Lemma~\ref{lem:direct-high-degree}, normalize
\(\mathcal B(V^{<3})+R\) by an additional gauge-fixed generator $
 T=T_5+T_6+\cdots.$
 Expanding
\(
 \mathcal B(V)=\exp(D_T)(\mathcal B(V^{<3})+R)
\)
at degree $8$, and then applying \(\Pi\), we get
\begin{equation}\label{eq:N3-degree-eight-expansion}
 [\mathcal B(V)-\mathcal B(V^{<3})]_8
 =
 \Pi[R]_8
 +\Pi D_{T_5}R_5
 +\frac12\Pi D_{T_5}^2H_2.
\end{equation}
where we used \(\Pi D_T\mathcal B(V^{<3})=0\).

The degree-$5$ equation for the Hamiltonian $R$ gives
\(
 D_{T_5}H_2=-R_5=-V_5
\),
hence
$T_5=-L^{-1}V_5$.
It also implies
$D_{T_5}^2H_2=-D_{T_5}R_5$.
Substitution into \eqref{eq:N3-degree-eight-expansion} yields
\[
 [\mathcal B(V)-\mathcal B(V^{<3})]_8
 =
 \Pi[R]_8+\frac12\Pi D_{T_5}V_5
 =
 \Pi[R]_8-\frac12\Pi D_{L^{-1}V_5}V_5.
\]
\end{proof}

\begin{definition}
For any $V\in \mathcal V_v^\sigma$, define the initial compatibility polynomial by
\begin{equation}\label{eq:Pquadratic}
 p_{3,V}(t):=B^{[4]}_8(V^{<4}(t))-B^{[4]}_8(V).
\end{equation}
where \(V^{<4}(t)\) is the affine family of truncated Taylor series defined in \eqref{eq:affine-line-family-jet}.
\end{definition}

\begin{lemma}\label{prop:quadratic-a14} 
The polynomial \(p_{3,V}(t)\) has degree at most two, and has $a_{14}(V)$ as a root. Moreover, the coefficients of $p_{3,V}$ depend rationally on \(j^6V\). In particular, its quadratic coefficient, denoted by \(\kappa_2(V)\), is given by
\begin{equation}\label{eq:kappa2}
 \kappa_2(V)
 =-\frac{\gamma_5}{2}\{L^{-1}\Phi_3(V),\Phi_3(V)\}_5, 
\end{equation}
which depends only on \(a_{12}(V)\) and \(a_{30}(V)\).
\end{lemma}

\begin{proof}
For fixed \(V\), we write
\(a_{14}=a_{14}(V)\),
\(\Phi_3=\Phi_3(V)\), \(p_3=p_{3,V}\) and
\(
 \mathcal T^\circ=\mathcal T_{V^{<3}}.
\)

Notice that \(B^{[4]}_8\) depends polynomially on the six-jet of the potential.  Since
\(
 (V_5(a_{14}),V_6(a_{14}))=(V_5,V_6),
\)
the potentials \(V^{<4}(a_{14})\) and \(V\) have the same degree-six
jet, one has
\(
 p_3(a_{14})=0.
\)

Applying Lemma~\ref{lem:direct-high-degree-N3} to \(V^{<4}(t)\) gives
\begin{align*}
 [\mathcal B(V^{<4}(t))-\mathcal B(V^{<3})]_8
 =\Pi[\exp(D_{S^\circ})(V_5(t)+V_6(t))]_8
  -\frac12\Pi D_{L^{-1}V_5(t)}V_5(t).
\end{align*}
Extracting the $\hbar^4$-coefficient yields
\begin{equation}\label{eq:base-B8}
 B^{[4]}_8(V^{<4}(t))-B^{[4]}_8(V^{<3})
 =\Pi\mathcal T^\circ(V_5(t),V_6(t))
  -\frac{\gamma_5}{2}\{L^{-1}V_5(t),V_5(t)\}_5.
\end{equation}
Subtracting \(R_3^{[4]}(V)=B_8^{[4]}(V)-B_8^{[4]}(V^{<3})\) from \eqref{eq:base-B8}, we get
\begin{equation}\label{eq:p3-full}
 p_3(t)
 =\Pi\mathcal T^\circ(V_5(t),V_6(t))
 -\frac{\gamma_5}{2}\{L^{-1}V_5(t),V_5(t)\}_5
 -R_3^{[4]}(V).
\end{equation}
The first term is linear in \(t\), whereas the fifth-bracket term is
quadratic.  Hence \(p_3\) has degree at most two, and its quadratic
coefficient is \eqref{eq:kappa2}.

\end{proof}

\begin{lemma}\label{lem:kappa2-nonzero}
For \(V\in\mathcal V_v^\sigma\),
the quadratic coefficient \(\kappa_2(V)\) of \(p_{3,V}\) is a nonzero rational
function of \(a_{30}(V)\), with \(a_{12}(V)\) as a parameter, and it
has no pole when \(\sigma a_{30}(V)>0\). 
Consequently, the set
\(
 \{V\in\mathcal V_v^\sigma:\kappa_2(V)\neq0\}
\)
is open and dense in \(\mathcal V_v^\sigma\).
\end{lemma}

\begin{proof}
By Lemma~\ref{lem:kernel-vector},
\(
 \Phi_3=\phi_{3,0}^*\tau^2x_1^5
 +\phi_{3,1}^*\tau x_1^3x_2^2+x_1x_2^4,
 \,\) with \(\tau=\frac{a_{12}v_2^2}{a_{30}\delta},
\)
and $\phi_{3,0}^*$ and $\phi_{3,1}^*$ are nonzero constants.
Proposition~\ref{prop:rational-affine-dependence} and
definition of $\kappa_2$ in \eqref{eq:kappa2} give the asserted rational dependence and absence of
poles.  As $|a_{30}|\to\infty$ within the fixed sign component, 
$\tau\to0$ and hence
\(
 \Phi_3\longrightarrow f:=x_1x_2^4.
\)
A direction computation shows that
\[
\lim_{|a_{30}|\to\infty}\kappa_2(V)=-\frac{1}{2}\gamma_5\{L^{-1}f,f\}_5
 =-\frac{1}{2}\gamma_5\binom51
 \bigl(\partial_{\xi_1}\partial_{\xi_2}^4L^{-1}f\bigr)
 \bigl(\partial_{x_1}\partial_{x_2}^4f\bigr)
 =-\frac{18v_2^4}
 {v_1\delta(v_1^2-16v_2^2)}.
\]

 Thus, for every fixed value of $a_{12}$,
$\kappa_2$ is a nonzero rational function of $a_{30}$ with no pole on
the half-line $\{\sigma a_{30}>0\}$.  Its zero set on this half-line is
therefore finite.  
This proves density in $\mathcal V_v^\sigma$ and
openness follows from continuity there.
\end{proof}

\bigskip



For \(V\in \mathcal V_v^\sigma\), consider the affine family of truncated Taylor series $ V^{<4}(t)$ defined by \eqref{eq:affine-line-family-jet}.
Let
\(\mathcal T_V^{(t)}=\mathcal T_{V^{<4}(t)}\) be its
\(\hbar^4\)-transfer operator, \((\widehat{V_7}(t),\widehat{V_8}(t))\) be the trial base point given by
\begin{equation}\label{eq:V7-V8-hat}
 \begin{pmatrix}
  \widehat{V_7}(t)\\
  \widehat{V_8}(t)
 \end{pmatrix}
 =\cA_4^{-1}\begin{pmatrix}
  B_{8}^{[2]}(V)-B_{8}^{[2]}(V^{<4}(t))\\[2mm]
  B_{8}^{[0]}(V)-B_{8}^{[0]}(V^{<4}(t))
 \end{pmatrix}.
\end{equation}
We define
\begin{align*}
 d_V(t)&:=\Pi\mathcal T_V^{(t)}(\Phi_4(V),\Psi_4(V)),
\\
 e_V(t)&:=R_{4,V}^{[4]}(t)
 -\Pi\mathcal T_V^{(t)}
 \bigl(\widehat{V_7}(t),\widehat{V_8}(t)\bigr).
\end{align*}
where $R_{4,V}^{[4]}(t)
 =B^{[4]}_{10}(V)-B^{[4]}_{10}(V^{<4}(t)).$
 Notice that
both \(d_V(t)\) and \(e_V(t)\) belong to
\(\spanop\{\Omega_1,\Omega_2\}\),  we can write them in the corresponding basis:
\[
 d_V(t)=d_{V,\Omega_1}(t)\Omega_1+d_{V,\Omega_2}(t)\Omega_2,
 \qquad
 e_V(t)=e_{V,\Omega_1}(t)\Omega_1+e_{V,\Omega_2}(t)\Omega_2.
\]
In particular,
\begin{equation}\label{eq:d-V-Omega2}
 d_{V,\Omega_2}(t)
 =d_4(V^{<4}(t))
 =\rho_4\mathcal T_V^{(t)}(\Phi_4(V),\Psi_4(V)).
\end{equation}
Define the compatibility determinant
\[
 \Xi_V(t)
 :=
 \det\bigl(e_V(t),d_V(t)\bigr)
 =
 e_{V,\Omega_1}(t)d_{V,\Omega_2}(t)
 -e_{V,\Omega_2}(t)d_{V,\Omega_1}(t).
\]

\begin{lemma}\label{lem:compatibility-polynomiality}
For every \(V\in\mathcal V_v^\sigma\),  the determinant functions \(\Xi_V(t)\) is a polynomial in
\(t\).  Its coefficients depend rationally on \(j^8V\)  with no poles on
\(\{a_{30}\neq0\}\); in particular, they depend continuously on
\(V\in\mathcal V_v^\sigma\).
\end{lemma}

\begin{proof}
Since \((V_5(t),V_6(t))\) is affine in \(t\), Lemmas
\ref{lem:finite-step-polynomial} and~\ref{lem:explicit-known-remainders}
show that \(B_{8}^{[r]}(V)-B_{8}^{[r]}(V^{<4}(t)), r=0,2\) are
polynomial in \(t\). Applying the
\(t\)-independent inverse \(\widetilde\cA_4^{-1}\), we get
\((\widehat V_7(t),\widehat V_8(t))\) is polynomial in \(t\).
Lemma~\ref{lem:T-explicit} and~\ref{lem:finite-step-polynomial} show that
\(\mathcal T_V^{(t)}\) and \(B_{10}^{[4]}(V^{<4}(t))\) are polynomial in
\(t\).  Therefore, \(\Xi_V(t)=\det(e_V(t),d_V(t))\) is a polynomial in \(t\).  Finally,
Proposition~\ref{prop:shifted-equation} with \(N=4\) shows that
\(B_{10}^{[4]}(V)\), and hence \(\Xi_V\), depends on \(j^8V\).
Proposition~\ref{prop:rational-affine-dependence} gives the stated
rational dependence.
\end{proof}

\begin{lemma}\label{prop:candidate-compatibility}
For every \(V\in \mathcal V_v^\sigma\), then $a_{14}(V)$ is a root of ~$\Xi_V(t)$.
\end{lemma}
\begin{proof}
Write \(H=H_V\) and
\(
 t_V^*=a_{14}(V),
 \,
 s_*=a_{16}(V)
\).  Then by Theorem~\ref{thm:block-affine-line},
\[
 (V_5,V_6)=(V_5(t_V^*),V_6(t_V^*)),\quad  (V_7,V_8)
 =(\widehat{V_7}(t_V^*),\widehat{V_8}(t_V^*))
  +s_*(\Phi_4,\Psi_4).
\]

Apply Proposition~\ref{prop:shifted-equation} with \(N=4\) to \(H\) and
the truncation \(H^{<4}(t_V^*)\).  Substituting the preceding affine
representation of \((V_7,V_8)\) to ~\eqref{eq:shifted-h4} gives
\[
 \Pi\mathcal T^{(t_V^*)}
 \bigl(\widehat{V_7}(t_V^*)+s_*\Phi_4,
       \widehat{V_8}(t_V^*)+s_*\Psi_4\bigr)
 =R_4^{[4]}(t_V^*).
\]
By the linearity of \(\mathcal T^{(t_V^*)}\), this is equivalent to
\[
 R_4^{[4]}(t_V^*)
 -\Pi\mathcal T^{(t_V^*)}
   (\widehat{V_7}(t_V^*),\widehat{V_8}(t_V^*))
 =s_*\Pi\mathcal T^{(t_V^*)}(\Phi_4,\Psi_4).
\]
The two sides of the above equation are precisely \(e(t_V^*)\) and \(s_*d(t_V^*)\), respectively.
Consequently,
\[
 \Xi(t_V^*)
 =\det\bigl(e(t_V^*),d(t_V^*)\bigr)
 =s_*\det\bigl(d(t_V^*),d(t_V^*)\bigr)
 =0.
\]
\end{proof}

\begin{definition}\label{def:normalized-compatibility}
For \(V\in \mathcal V_v^\sigma\) with
\(\kappa_2(V)\neq0\),
put \(t_V^*=a_{14}(V)\), and let \(t_V^\sharp\) be the other root of
\(p_{3,V}\), counted with multiplicity.
Define the conjugate-root obstruction by
\[
 \chi^\sharp(V):=\chi_V(t_V^\sharp).
\]
where $\chi_V(t)$ is the unique polynomial such that \(\Xi_V(t)=(t-a_{14}(V))\chi_V(t)\).
\end{definition}

\begin{lemma}\label{lem:a70-transversality}
Let \(V\in \mathcal V_v^\sigma\) satisfy
\(\kappa_2(V)\neq0\).  Then \(\chi^{\sharp}(V)\) depends rationally on
\(j^8V\) and has no poles.  Furthermore,
\begin{equation}\label{eq:chi-sharp-derivative}
 \left.\frac{\dd}{\dd\varepsilon}\right|_{\varepsilon=0}
 \chi^\sharp(V+\varepsilon x_1^7)
 =\frac{42\phi_{3,0}^*\tau^2}{v_1}
   d_{V,\Omega_2}(t_V^\sharp).
\end{equation}
where $\tau=a_{12}(V)v_2^2/(a_{30}(V)\delta)$ and
$\phi_{3,0}^*$ is the nonzero constant in \eqref{eq:phi}.
\end{lemma}

\begin{proof}
The asserted rational dependence of \(\chi^\sharp(V)\) follows from that of \(\Xi_V(t)\) and \(p_{3,V}(t)\) in
Lemma~\ref{prop:quadratic-a14} and ~\ref{lem:compatibility-polynomiality}.  
Fix \(V\in\mathcal V_v^\sigma\) with \(\kappa_2(V)\neq0\).
Write \(H=H_V\), and set
\[
 V_{\varepsilon}=V+\varepsilon x_1^7,
 \qquad
 H_{\varepsilon}=H_{V_{\varepsilon}}.
\]
The perturbation does not change
$V_3,\ldots,V_6$ hence $B^{[4]}_8$.
Thus \(p_{3,V_\varepsilon}=p_{3,V}\), so its two roots remain
\(t_V^*\) and \(t_V^\sharp\), and
\(d_{V_\varepsilon}(t)=d_V(t)\).  We write these fixed quantities simply
as \(p_3(t)\) and \(d(t)\) and denote
\(\Xi_\varepsilon=\Xi_{V_\varepsilon}\), 
\(\chi_\varepsilon=\chi_{V_\varepsilon}\), $\mathcal T^{(t)}= \mathcal T_{V_{\epsilon}}^{(t)}$.

Since \(V_\varepsilon=V+\varepsilon x_1^7\) has the same six-jet as \(V\), the lower degree remainders and \(\widetilde{\mathcal A}_4\) are independent of \(\varepsilon\). Hence the degree-eight block equations give
\(\begin{pmatrix}B_8^{[2]}(V_\varepsilon)-B_8^{[2]}(V)\\ B_8^{[0]}(V_\varepsilon)-B_8^{[0]}(V)\end{pmatrix}=\varepsilon\widetilde{\mathcal A}_4\begin{pmatrix}x_1^7\\0\end{pmatrix},\)
Differentiating \eqref{eq:V7-V8-hat} for $V_{\epsilon}$ gives
\[
 \left.\frac{\dd}{\dd\varepsilon}\right|_{0}
 \widehat{V_{7,\varepsilon}}(t)=x_1^7,
 \qquad
 \left.\frac{\dd}{\dd\varepsilon}\right|_{0}
 \widehat{V_{8,\varepsilon}}(t)=0.
\]
 Since $V^{<4}(t)$ is independent of \(\varepsilon\), we have
 \(
  R_{4,\varepsilon}^{[4]}(t)-R_4^{[4]}(t)
  =
  B_{10}^{[4]}(V_{\varepsilon})-B_{10}^{[4]}(V).
 \)
 Applying \eqref{eq:shifted-h4} to
\(H_{\varepsilon}\) and \(H\) with the same base $H^{<4}(t_V^*)$, one obtains
\(
 B_{10}^{[4]}(V_{\varepsilon})-B_{10}^{[4]}(V)
 =
 \varepsilon\Pi\mathcal T^{(t_V^*)}(x_1^7,0).
\)
Thus
\(
 R_{4,\varepsilon}^{[4]}(t)
 =
 R_4^{[4]}(t)
 +\varepsilon\Pi\mathcal T^{(t_V^*)}(x_1^7,0).
\)
By definition of \(e_\varepsilon(t)\),
\begin{equation}\label{eq:e-variation}
 \left.\frac{\dd}{\dd\varepsilon}\right|_{0}e_\varepsilon(t)
 =\Pi\bigl(\mathcal T^{(t_V^*)}-\mathcal T^{(t)}\bigr)(x_1^7,0).
\end{equation}
Only the $N_5$ term in $\mathcal T^{(t)}$ depends on $t$.  The
degree-five homological equation gives
\(
 S_5(t)=-L^{-1}\bigl(V_5(t)+K_5\bigr),
\)  Hence
\[
 \bigl(\mathcal T^{(t_V^*)}-\mathcal T^{(t)}\bigr)(x_1^7,0)
 =\gamma_5(t-t_V^*)\{L^{-1}\Phi_3,x_1^7\}_5.
\]
Since \(x_1^7\) depends only on \(x_1\), one has
\(
\{L^{-1}\Phi_3,x_1^7\}_5=(\partial_{\xi_1}^5L^{-1}\Phi_3)(\partial_{x_1}^5x_1^7).
\) As \(L^{-1}\) preserves the degree in each oscillator pair, the \(x_1^3x_2^2\)- and \(x_1x_2^4\)-components of \(L^{-1}\Phi_3\) are annihilated by \(\partial_{\xi_1}^5\); hence only the \(x_1^5\)-component contributes.
Using
\(
 \partial_{\xi_1}^5L^{-1}(x_1^5)=\frac{64}{v_1},
 \,
 \Pi(x_1^2)=\frac12\Omega_1,
\)
we obtain from \eqref{eq:e-variation} that
\[
 \left.\frac{\dd}{\dd\varepsilon}\right|_{0}e_\varepsilon(t)
 =\frac{42\phi_{3,0}^*\tau^2}{v_1}(t-t_V^*)\Omega_1.
\]
As $d(t)$ is independent of $\varepsilon$, it follows that the derivative of the determinant $\Xi_\varepsilon(t)=\det(e_\varepsilon(t),d(t))$ at $\varepsilon=0$ is
\[
 \left.\frac{\dd}{\dd\varepsilon}\right|_{0}\Xi_\varepsilon(t)
 =\frac{42\phi_{3,0}^*\tau^2}{v_1}(t-t_V^*)d_{\Omega_2}(t).
\]
Since \(t_V^*\) and $t^\sharp$ is fixed and
\(\Xi_\varepsilon(t)=(t-t_V^*)\chi_\varepsilon(t)\), division
by \(t-t_V^*\) gives
 \eqref{eq:chi-sharp-derivative}.
\end{proof}

\begin{proposition}\label{prop:normalized-compatibility-generic}
The set
\[
 \{V\in\mathcal V_v^\sigma:
   \kappa_2(V)\neq0,\ \chi^\sharp(V)\neq0\}
\]
is open and dense in \(\mathcal V_v^\sigma\).
\end{proposition}
\begin{proof}
The openness follows from the continuity of \(\kappa_2\) and \(\chi^\sharp\) on the domain \(\{\kappa_2\neq0\}\).
To prove density, by Lemma \ref{lem:a70-transversality} it suffices to show the density of the set $\{V\in\mathcal V_v^\sigma:\kappa_2(V)\neq0,\ d_{V,\Omega_2}(t_V^\sharp)\neq0\}$.

Let \(\mathcal U:=\{V\in\mathcal V_v^\sigma:\kappa_2(V)\neq0\}\), which is
open and dense by Lemma~\ref{lem:kappa2-nonzero}.
Fix \(V\in\mathcal U\), and put
\(t_V^*=a_{14}(V)\).  Define \(V^\sharp\) by leaving every homogeneous
term  other than \(V_5,V_6\) unchanged in $V$ and replacing
\(
 (V_5,V_6)
 =(\widehat{V_5}(V)+t_V^*\Phi_3(V),
   \widehat{V_6}(V)+t_V^*\Psi_3(V))
\)
by
\[
 (V_5^\sharp,V_6^\sharp)
 =(\widehat{V_5}(V)+t_V^\sharp\Phi_3(V),
   \widehat{V_6}(V)+t_V^\sharp\Psi_3(V)).
\]
It follows by definition of \(V^\sharp\) that
\(
 (V^\sharp)^{<4}(t)=V^{<4}(t).
\)
Since \(t_V^\sharp\) is a root of \(p_{3,V}\), \eqref{eq:Pquadratic}
gives
\(
 B_8^{[4]}(V^{<4}(t_V^\sharp))=B_8^{[4]}(V).
\)
The potentials \(V^\sharp\) and \(V^{<4}(t_V^\sharp)\) have the same
six-jet, hence
\(
 B_8^{[4]}(V^\sharp)
 =B_8^{[4]}(V).
\)
Thus $p_{3,V^\sharp}(t)=p_{3,V}(t)$, so $t_V^*$ and \(t_V^\sharp\) are
the two roots of \(p_{3,V^\sharp}\).  It follows that
\((V^\sharp)^\sharp=V\).  Since $p_{3,V}$ depends continuously on $V$, so is the root $t_V^\sharp$, thus the map
\(V\mapsto V^\sharp\) is a continuous involution of \(\mathcal U\).

Notice that
\(
 d_{V,\Omega_2}(t_V^\sharp)
 =d_4(V^{<4}(t_V^\sharp))
 =d_4(V^\sharp).
\)
To prove density, for \(V\in\mathcal U\) with
\(a_{12}(V)\neq0\), consider 
\( W_\varepsilon=V^\sharp+\varepsilon x_1^5,
 \,
 V_\varepsilon=W_\varepsilon^\sharp.
\)
The perturbation does not change \(\kappa_2\), so
\(W_\varepsilon\in\mathcal U\), and the continuity and involutivity of
the sharp map imply
\(
 V_\varepsilon\rightarrow(V^\sharp)^\sharp=V
 \,\text{as }\varepsilon\rightarrow0.
\)
Hence
\[
 d_{V_\varepsilon,\Omega_2}(t_{V_\varepsilon}^\sharp)
 =d_4(W_\varepsilon)
 =d_4(V^\sharp)+\varepsilon m_4,
 \qquad
 m_4:=\frac{\partial d_4(V^\sharp)}{\partial a_{50}}\neq0.
\]
Thus every
neighborhood of \(V\) contains some \(V_\varepsilon\) for which
\(d_{V_\varepsilon,\Omega_2}(t_{V_\varepsilon}^\sharp)\neq0\).  The
condition is therefore dense.

\end{proof}

\begin{proof}[Proof of Theorem~\ref{thm:recover-a14}]
Define
\[
 \mathcal G_{v,\sigma}^{\alpha}
 :=\{V\in\mathcal V_{v,\sigma}^{\alpha}:\kappa_2(V)\neq0,\ \chi^\sharp(V)\neq0\}
 \cap\bigcap_{N\ge4}\{d_N(V)\neq0\}.
\]
By Proposition~\ref{prop:d-nonzero} and Proposition ~\ref{prop:normalized-compatibility-generic}, this set is a dense \(G_\delta\) subset of
\(\mathcal V_{v,\sigma}^{\alpha}\). 

Let \(V\in\mathcal G_{v,\sigma}^{\alpha}\), and suppose that
\(W\in\mathcal V_{v,\sigma}^{\alpha}\) has the same
\((B^{[0]},B^{[2]},B^{[4]})\) as \(V\).  The common degree-four data,
\(a_{12}=\alpha\), and the prescribed sign of \(a_{30}\) determine the
same \(V_3,V_4\).  Theorem~\ref{thm:block-affine-line}, applied to the
common \(B_6^{[2]},B_6^{[0]}\), shows that $V^{<4}(t)=W^{<4}(t)$.
The common \(B_8^{[4]}\) then gives
\(p_{3,V}=p_{3,W}\), while the common 
\(B_{10}^{[4]}\) give \(\Xi_V=\Xi_W\).

Because \(\kappa_2(V)\neq0\), the common polynomial \(p_{3,V}=p_{3,W}\)
has roots \(t_V^*=a_{14}(V)\) and \(t_V^\sharp\), counted with
multiplicity, and \(a_{14}(W)\) is one of them.  If the roots coincide,
then \(a_{14}(W)=a_{14}(V)\).  If they are distinct,
one has
\(\Xi_V(a_{14}(W))=\Xi_W(a_{14}(W))=0\), whereas
\[
 \Xi_V(t_V^\sharp)
 =(t_V^\sharp-t_V^*)\chi^\sharp(V)\neq0.
\]
Thus \(a_{14}(W)=a_{14}(V)\), so the degree-six jets of \(V\) and \(W\)
are the same.

For \(N\ge4\), assume inductively that the lower terms agree.
Corollary~\ref{cor:scalar-edge-equation} then gives the same scalar edge
equation for both potentials.  Since \(d_N\) depends only on their
common five-jet and \(d_N(V)\neq0\), it determines the same edge
parameter; Theorem~\ref{thm:block-affine-line} then determines the same
pair \((V_{2N-1},V_{2N})\).  Hence \(W=V\) by induction.  The same
steps give the asserted successive determination.
\end{proof}

We end this section with a lemma that will be used later.

\begin{lemma}\label{lem:p3-simple-root-perturbation}
For every \(\alpha\neq0\), the set
\[
 \left\{V\in\mathcal V_{v,\sigma}^{\alpha}:
 \kappa_2(V)\neq0,\quad
 p_{3,V}'(a_{14}(V))\neq0\right\}
\]
is open and dense in \(\mathcal V_{v,\sigma}^{\alpha}\).
\end{lemma}

\begin{proof}
It is enough to show the density of the condition $\{p_{3,V}'(a_{14}(V))\neq0\}$.  Fix \(V\in \mathcal V_{v,\sigma}^{\alpha}\) and $\kappa_2(V)\neq 0$, write
\(t_0=a_{14}(V)\).  For \(\varepsilon\in\mathbb R\), let
\(V_\varepsilon\) be obtained by replacing
\((V_5,V_6)\) with
\((V_5,V_6)+\varepsilon(\Phi_3,\Psi_3)\), leaving all other
homogeneous terms fixed.  The normalization
\(\phi_{3,2}=1\) means that the \(x_1x_2^4\)-coefficient of
\(\Phi_3\) is one, hence
\(
 a_{14}(V_\varepsilon)=t_0+\varepsilon.
\)
Moreover, \((\Phi_3,\Psi_3)\in\ker\cA_3\), so the perturbation preserves
\(B_6^{[2]},B_6^{[0]}\). Thus we see that
\(
 V_\varepsilon^{<4}(t)=V^{<4}(t).
\)

Set
\(
 F(t):=B_8^{[4]}(V^{<4}(t)).
\)
By Proposition~\ref{prop:quadratic-a14}, \(F\) is quadratic with
leading coefficient \(\kappa_2(V)\neq0\).  Since
\(j^6V_\varepsilon=j^6(V^{<4}(t_0+\varepsilon))\) and
\(B_8^{[4]}\) depends only on the six-jet, one get
\(
 B_8^{[4]}(V_\varepsilon)=F(t_0+\varepsilon).
\)
Thus, by definition,
\(
 p_{3,V_\varepsilon}(t)=F(t)-F(t_0+\varepsilon),
\)
and consequently
\[
 p_{3,V_\varepsilon}'(a_{14}(V_\varepsilon))
 =F'(t_0+\varepsilon).
\]
The derivative \(F'\) is a nonzero affine polynomial, it has only finitely many zeros, this proves the
density.
\end{proof}

\section{Separation and determination of the initial cubic pair}\label{sec:initial-pair}

We now fix only the nonresonant frequencies; neither the
orientation nor the coefficients \(a_{12}\) and \(a_{14}\) are prescribed.
Since the QBNF is invariant under \(V(x)\mapsto V(-x)\), the cubic pair
$(a_{30},a_{12})$ can be determined at best up to sign.

Write \(J_{10}:=J_{10,v}\).  Thus \(J_{10}\) is the space of
\(\mathbb Z_2\)-symmetric real ten-jets
\[
 X=X(V)=(V_2,V_3,\ldots,V_{10}).
\]
where \(V_2\) is fixed by \(v\). Thus \(J_{10}\) is identified with \(\R^{32}\).
\begin{definition}
Define the polynomial map $\cI_{10}:J_{10}\to\R^{30}$ by
\begin{align}
 \cI_{10}(X):=\bigl(&B^{[2]}_4,B^{[0]}_4;
 B^{[2]}_6,B^{[0]}_6,B^{[4]}_8;
 \notag\\
 &B^{[2]}_8,B^{[0]}_8,\rho_4(B^{[4]}_{10});
 B^{[2]}_{10},B^{[0]}_{10}\bigr)(X).
 \label{eq:I10}
\end{align}
\end{definition}
\subsection{Separation of the cubic pair}

We separate the cubic pair \((a_{30},a_{12})\) by the following
finite dimensional argument.  Once the map \(\cI_{10}\) is shown to have
generic rank \(30\), its generic
fibers are two-dimensional, and the incidence relation
\(\cI_{10}(X_*)=\cI_{10}(X)\) has dimension \(34\).  When the
cubic pairs of \(X_*\) and \(X\) are neither equal nor negatives of one another,
the three components of
\(B_{12}^{[4]}(X_*)-B_{12}^{[4]}(X)\) vary independently.  Their
vanishing cuts the incidence relation to dimension \(31\).  Since
\(31<32\), its projection to the reference jet is nowhere dense.

\begin{lemma}\label{lem:generic-fiber-I10}
The map
\(
 \cI_{10}:J_{10}\longrightarrow\R^{30}
\)
has generic rank $30$.  More precisely, the set
\begin{equation}\label{eq:Ufib}
 \begin{aligned}
 \cU_{\mathrm{fib}}:=\bigl\{X(V)\in J_{10}:\;&
 a_{30}(V)\neq0,\, a_{12}(V)\neq0,\,
 \kappa_2(V)\neq0,\\
 &p_{3,V}'(a_{14}(V))\neq0,\, d_4(V)\neq0\bigr\}
 \end{aligned}
\end{equation}
is open and dense, and \(\cI_{10}\) is a submersion at every point of
\(\cU_{\mathrm{fib}}\).
\end{lemma}

\begin{proof}
Fix \(X=X(V)\in\cU_{\mathrm{fib}}\), and write
\[
 X_1=(V_3,V_4),\quad X_2=(V_5,V_6),\quad
 X_3=(V_7,V_8),\quad X_4=(V_9,V_{10}),
\]
and let $Y_1,\ldots,Y_4$ denote the four groups of target coordinates in
\eqref{eq:I10}.  Weyl-degree triangularity gives
\[
 D\cI_{10}=(D_{ij})_{1\leq i,j\leq4},
 \qquad D_{ij}=0\quad\text{for }j>i.
\]
where $D_{ij}=D_{X_j}Y_i$.
Denote $A_N=[\mathcal{A}_N]$ and 
 the row covectors
\[
 q_3:=D_{(V_5,V_6)}B_8^{[4]},
 \qquad
 q_4:=D_{(V_7,V_8)}\bigl(\rho_4(B_{10}^{[4]})\bigr).
\]
 For \(D_{11}\), order the columns as
\((a_{30},a_{12};V_4)\) and the rows as
\((B_4^{[2]};B_4^{[0]})\).  Then the four diagonal blocks of $D\mathcal I_{10}$ are
\[
 \begin{split}
 D_{11}
 &=\begin{pmatrix}
   \dfrac{a_{30}}{v_1}&
   -\dfrac{a_{12}v_2^2}{v_1\delta}&0_{1\times3}\\&
   \left[D_{V_3}\left(\dfrac12\Pi\{S_3,V_3\}\right)\right]
   &[\mathcal E_2]
  \end{pmatrix},\\
 D_{22}&=\begin{pmatrix}A_3\\ q_3\end{pmatrix},
 \qquad
 D_{33}=\begin{pmatrix}A_4\\ q_4\end{pmatrix},
 \qquad
 D_{44}=A_5.
 \end{split}
\]
Their respective sizes are \(4\times5\), \(7\times7\), \(9\times9\),
and \(10\times11\).  It therefore suffices to prove that each has full
row rank.

Firstly, $D_{11}$ has full rank $4$ since  the square minor
with diagonal blocks \((a_{30}/v_1,[\mathcal E_2])\) is invertible.
Write the kernel of $A_N$ as
 $\eta_N=(\Phi_N,\Psi_N)$. The rank of $A_N$ is $2N$ and
the additional row $q_N$ raises the rank by one if
$q_N\eta_N^{\mathsf T}\neq0$, $N=3,4$.
Notice that $\frac{d}{dt} V^{<N}(t)=\eta_N$, thus
\begin{equation}\label{eq:qN-etaN}
 q_N\eta_N^{\mathsf T}
 =\left.\frac{d}{dt}B_{2N+2}^{[4]}(V^{<N}(t))\right|_{t=a_{1,2N-2}(V)}.
\end{equation}
For $N=3$, the right-hand side of \eqref{eq:qN-etaN} is the derivative of the quadratic polynomial \(p_{3,V}\) at its root \(a_{14}(V)\).
Thus \(\rank D_{22}=7\) whenever
\(p_{3,V}^{\prime}(a_{14}(V))\neq0\).

For $N=4$, apply Proposition~\ref{prop:shifted-equation} with $N=4$ to $V^{<5}(t)$, one has
\[
 \rho_4\bigl(B_{10}^{[4]}(V^{<5}(t))-B_{10}^{[4]}(V^{<4})\bigr)
 =\rho_4\mathcal T_V(V_7(t),V_8(t)).
\]
thus
\(
 q_4\eta_4^{\mathsf T}
 =\rho_4\mathcal T_V(\Phi_4,\Psi_4)=d_4(V).
\)
since $\mathcal T_V$ is linear.
Consequently, $\rank D_{33}=9$ whenever $d_4(V)\neq0$.

Finally, restrict the columns of $D_{44}$ to
$\widetilde{\mathsf O}_9\oplus\mathsf E_{10}$.  The resulting square
matrix is $[\widetilde\cA_5]$, which is invertible by
Proposition~\ref{prop:block}.  Therefore $\rank D_{44}=10$.

We conclude that $\rank D\cI_{10}$ has full rank $30$ at every point of \(\cU_{\mathrm{fib}}\).  All the quantities in the
nonvanishing conditions in \eqref{eq:Ufib} are continuous on their
common domain \(\{a_{30}\neq0\}\), so \(\cU_{\mathrm{fib}}\) is open
in \(J_{10}\).  Density of $\cU_{\mathrm{fib}}$ follows easily from Lemma~\ref{lem:p3-simple-root-perturbation} and Proposition~\ref{prop:d-nonzero}.


The rank computation proves that \(\cI_{10}\) is a submersion throughout
\(\cU_{\mathrm{fib}}\).
\end{proof}

\begin{remark}\label{rem:augmented-I10}
On \(\cU_{\mathrm{fib}}\), the map
\(
 (\cI_{10},a_{12},a_{30}):J_{10}\longrightarrow\R^{32}
\)
has rank \(31\), while the map
\(
 (\cI_{10},a_{12},a_{18}):J_{10}\longrightarrow\R^{32}
\)
is a local diffeomorphism.  Indeed, relative to the preceding block
decomposition, the new diagonal blocks are respectively
\[
 \widehat D_{11}:=
 \begin{pmatrix}D_{11}\\D_{X_1}a_{12}\end{pmatrix},
 \qquad
 \widehat D_{44}:=
 \begin{pmatrix}D_{44}\\D_{X_4}a_{18}\end{pmatrix}.
\]
The first is invertible by the same minor argument used for \(D_{11}\),
and \(D_{X_1}a_{30}\) adds no further rank.  The second is invertible
because
\(\ker D_{44}=\operatorname{span}\{(\Phi_5,\Psi_5)\}\) and the
normalization gives \(D_{X_4}a_{18}(\Phi_5,\Psi_5)=1\).
\end{remark}

\begin{lemma}\label{lem:M3-explicit}
For $X\in J_{10}$, the restriction of the operator
$M_3(X)=M_{S_3(X)}$ to potential polynomials is given by
\begin{equation}\label{eq:M3X}
 M_3(X)
 =\mu(X)\partial_{x_1}^3
  +\nu(X)\partial_{x_1}\partial_{x_2}^2,
 \qquad \text{where}\,\,\,
 \mu(X):=\frac{a_{30}(X)}{6v_1},
 \,\,
 \nu(X):=-\frac{a_{12}(X)v_2^2}{2v_1\delta}.
\end{equation}
\end{lemma}

\begin{proof}
Let $F=F(x_1,x_2)$ be a potential polynomial.  By the definition of
$M_A$ in \eqref{eq:PMN},
\[
M_{S_3}(F)=\bigl(D_{S_3(X)}F\bigr)^{[2]}
 =-\gamma_3\{S_3(X),F\}_3.
\]
Since $F$ is independent of $\xi$, formula~\eqref{eq:S3-formula} gives
\[
 -\gamma_3\{S_3(X),F\}_3
 =\left(
   \frac{a_{30}(X)}{6v_1}\partial_{x_1}^3
   -\frac{a_{12}(X)v_2^2}{2v_1\delta}
    \partial_{x_1}\partial_{x_2}^2
  \right)F.
\]
This is exactly \eqref{eq:M3X}.
\end{proof}

\begin{lemma}\label{lem:rank-three}
Let \(X_*,X\in J_{10}\) and assume that
\(
 (a_{30}(X),a_{12}(X))
\neq \pm(a_{30}(X_*),a_{12}(X_*)).
\)
Then the map
\begin{equation}\label{eq:TXX}
 \mathsf E_{10}\longrightarrow\mathsf N_{12}^{[4]},
 \qquad
 G\longmapsto\frac12\Pi(M_3(X_*)^2-M_3(X)^2)G,
\end{equation}
is surjective.
\end{lemma}

\begin{proof}
Write
\(
 (\mu_*,\nu_*)=(\mu(X_*),\nu(X_*)),
 \,
 (\mu,\nu)=(\mu(X),\nu(X)).
\)
Expanding the constant-coefficient operator
$T:=M_3(X_*)^2-M_3(X)^2$, one get that $T=P(\partial_{x_1},\partial_{x_2})$ where
$P$ is the degree-six polynomial
\[
 P(x_1,x_2)
 =(\mu_*^2-\mu^2)x_1^6
 +2(\mu_*\nu_*-\mu\nu)x_1^4x_2^2
 +(\nu_*^2-\nu^2)x_1^2x_2^4.
\]
By assumption, \((\mu,\nu)\neq\pm(\mu_*,\nu_*)\), so \(P\) is not identically zero.

For homogeneous real polynomials of a fixed degree, consider the Fischer
inner product
\begin{equation*}
 \langle f,g\rangle_{\mathrm F}
 :=\left.f(\partial_x)g(x)\right|_{x=0},
\qquad
 \partial_x=(\partial_{x_1},\partial_{x_2}).
\end{equation*}
Equivalently,
\(
 \langle x^\alpha,x^\beta\rangle_{\mathrm F}
 =\alpha!\,\delta_{\alpha\beta}.
\)
The inner product is positive definite and differentiation is
adjoint to multiplication:
\(
\langle\partial_{x_j}f,g\rangle_{\mathrm F}
 =\langle f,x_jg\rangle_{\mathrm F}.
\)
It follows that the adjoint of
$T=P(\partial_x):\mathsf E_{10}\to\mathsf E_4$ is multiplication by
$P$:
\begin{equation*}
 \langle TG,H\rangle_{\mathrm F}
 =\langle G,PH\rangle_{\mathrm F},
 \qquad G\in\mathsf E_{10},\ H\in\mathsf E_4.
\end{equation*}
Since $P\neq0$, the adjoint of $T$ is injective,  which implies that
$T:\mathsf E_{10}\to\mathsf E_4$ is surjective.  Finally,
Proposition~\ref{prop:even-block} says that
\(
 \cE_2=\Pi|_{\mathsf E_4}:\mathsf E_4\to\mathsf N_4^{[0]}
\)
is an isomorphism.  As coefficient spaces,
\[
 \mathsf N_4^{[0]}=\mathsf N_{12}^{[4]}
 =\spanop\{\Omega_1^2,\Omega_1\Omega_2,\Omega_2^2\}.
\]
Composing \(T\) with this isomorphism proves
the surjectivity of the map in \eqref{eq:TXX}.
\end{proof}

\begin{lemma}\label{lem:separating-submersion}
Define
\begin{equation}\label{eq:separating-incidence}
\begin{aligned}
 \cR:=\{(X_*,X)\in\cU_{\mathrm{fib}}\times J_{10}:\;&
 \cI_{10}(X_*)=\cI_{10}(X),\\
 &(a_{30}(X),a_{12}(X))\neq 
 \pm(a_{30}(X_*),a_{12}(X_*))\bigr\}.
\end{aligned}
\end{equation}
Then $\cR$ is a smooth manifold of dimension $34$, and the map
\[
 \Theta:\cR\longrightarrow\mathsf N_{12}^{[4]},
 \qquad
 \Theta(X_*,X)=B^{[4]}_{12}(X_*)-B^{[4]}_{12}(X),
\]
is a submersion.
\end{lemma}

\begin{proof} By Remark 6.3, the map \(X\mapsto(a_{30}(X),a_{12}(X))\) is nonconstant on every local \(\mathcal I_{10}\)-fiber through \(X_*\in\mathcal U_{\mathrm{fib}}\), thus
\(\mathcal R\neq\varnothing\).
Lemma~\ref{lem:generic-fiber-I10} shows that \(\cI_{10}\) is a
submersion at \(X_*\).  Hence the map
\(
 \cU_{\mathrm{fib}}\times J_{10}\longrightarrow\R^{30},
 \,
 (X_*,X)\longmapsto\cI_{10}(X_*)-\cI_{10}(X)
\)
.  Its zero set has dimension $2\cdot32-30=34$, and
 $\cR$ is an open subset of this zero set.   Therefore
\(
 \dim\cR=34.
\)

Fix $G\in\mathsf E_{10}$.  For $Y\in\{X_*,X\}$, let $Y_s$ be obtained
from $Y$ by replacing $V_{10}(Y)$ with $V_{10}(Y)+sG$. By
Proposition~\ref{prop:shifted-equation},
\[
 B^{[4]}_{12}(Y_s)-B^{[4]}_{12}(V^{<5}(Y))
 =
 \Pi\mathcal T_Y\bigl(V_9(Y),V_{10}(Y)+sG\bigr).
\]
By the formula \eqref{eq:L-operator}, one has
\[
 \left.\frac{d}{ds}B^{[4]}_{12}(Y_s)\right|_{s=0}
 =\frac12\Pi M_3(Y)^2G.
\]

Apply this variation simultaneously to $\cI_{10}(X_*)$ and $\cI_{10}(X)$.
Equation~\eqref{eq:block-matrix} and
Proposition~\ref{prop:shifted-equation} with $N=4$ show that only
$B^{[0]}_{10}$ among the components of $\cI_{10}$ depends on
$V_{10}$, and its variation is $\cE_5G$.  Thus both values of
$\cI_{10}$ take the same change. So $(G,G)$ is tangent to $\cR$.
As
\[
 D\Theta_{(X_*,X)}(G,G) =\frac12\Pi(M_3(X_*)^2-M_3(X)^2)G.
\]
By Lemma~\ref{lem:rank-three} the right-hand side is surjective as $G$
varies.  Hence $\Theta$ is a submersion.
\end{proof}

\begin{proposition}\label{prop:separate-cubic}
There is a dense \(G_\delta\) subset
$\cU_{\mathrm{sep}}\subset J_{10}$ such that if $X_*\in\cU_{\mathrm{sep}}$ and
$X\in J_{10}$ satisfy
\begin{equation}\label{eq:same-finite-data}
 \cI_{10}(X)=\cI_{10}(X_*),
 \qquad
 B^{[4]}_{12}(X)=B^{[4]}_{12}(X_*),
\end{equation}
then
\(
 (a_{30}(X),a_{12}(X))
 =\pm(a_{30}(X_*),a_{12}(X_*)).
\)
\end{proposition}

\begin{proof}
By Lemma~\ref{lem:separating-submersion}, the zero set
\(
 \cZ:=\Theta^{-1}(0)
\)
is either empty or a smooth submanifold of \(\cR\) of codimension
\(3\).  In either case,
\(
 \dim\cZ\leq34-3=31.
\)

Consider the projection onto the first jet,
\(
 \pi_1:\cZ\longrightarrow J_{10},
 \,
 \pi_1(X_*,X)=X_*.
\)
Since
\(
 \dim\cZ<\dim J_{10},
\)
 every point of \(\cZ\) is critical for \(\pi_1\). By Sard's lemma,
 \(\pi_1(\cZ)\) has measure zero in \(J_{10}\). Let \(\cZ=\bigcup_{j\geq1}K_j\) be a compact exhaustion of \(\cZ\).  For each \(j\),
the set \(\pi_1(K_j)\) is compact and of measure zero,  therefore is nowhere dense.  Thus
\(J_{10}\setminus\pi_1(\cZ)\) is a dense \(G_\delta\) set.  Set
\begin{equation*}
 \cU_{\mathrm{sep}}
 :=\cU_{\mathrm{fib}}
   \setminus\pi_1(\cZ).
\end{equation*}
Lemma~\ref{lem:generic-fiber-I10} shows that $\cU_{\mathrm{fib}}$ is
open and dense.  Thus \(\cU_{\mathrm{sep}}\) is a dense \(G_\delta\) subset
of \(J_{10}\).  If
\(X_*\in\cU_{\mathrm{sep}}\) and a jet \(X\) satisfying
\eqref{eq:same-finite-data} and \((a_{30}(X),a_{12}(X))\neq\pm(a_{30}(X_*),a_{12}(X_*))\), then \((X_*,X)\in\cZ\), contradicting
\(X_*\notin\pi_1(\cZ)\). Thus we conclude.
\end{proof}

\begin{remark}
Alternatively, the semialgebraicity of the sets and maps above, as well as the
inequality \(\dim\cZ\leq31<32=\dim J_{10}\) show that
\(\cU_{\mathrm{fib}}\setminus\overline{\pi_1(\cZ)}\) is open and dense and
has the same separation property.
\end{remark}

\subsection{Global determination from finite compatibility}

We now encode the finite separation argument by a polynomial compatibility
system.  Write a supplied low-order datum in the coordinates of
\eqref{eq:I10} as
\[
 \mathbf Y=\bigl(q,Y_2^{[0]};
 Y_3^{[2]},Y_3^{[0]},\mathsf b_8;
 Y_4^{[2]},Y_4^{[0]},\mathsf b_{10};
 Y_5^{[2]},Y_5^{[0]}\bigr),
\]
where
\(
 q=b_{1,0,0},\,
 \mathsf b_8=B_8^{[4]},\,
 \mathsf b_{10}=\rho_4(B_{10}^{[4]}).
\)
The remaining datum is written
\[
 \mathbf Z:=B_{12}^{[4]}(X_*)=Z_{20}\Omega_1^2
 +Z_{11}\Omega_1\Omega_2+Z_{02}\Omega_2^2.
\]
\begin{definition}
We say that the pair $(\mathbf Y,\mathbf Z)$ is \emph{admissible} if there
exists a jet
\(X_*\in\cU_{\mathrm{sep}}\subset J_{10}\) such that
\(
 \cI_{10}(X_*)=\mathbf Y,
 \,
 B_{12}^{[4]}(X_*)=\mathbf Z. 
\)
\end{definition}

Given an admissible pair $(\mathbf Y,\mathbf Z)$, let $X=j^{10}V\in J_{10}$ be a
jet such that $\cI_{10}(X)=\mathbf Y$ and
$B_{12}^{[4]}(X)=\mathbf Z$.  Denote the cubic pair of \(X\) by
\(
 a:=a_{30}(X),\, s:=a_{12}(X).
\)

For $N=2$, the $\hbar^2$ part of the degree-four equation
$B_4^{[2]}(X)=q$ gives
\begin{equation}\label{eq:global-cubic-conic}
 \mathsf Q(a,s)
 :=a^2-\frac{v_2^2}{\delta}s^2-2v_1q=0.
\end{equation}
The $\hbar^0$ part of the degree-four equation
$B_4^{[0]}(X)=Y_2^{[0]}$ gives
\begin{equation}\label{eq:finite-initial-quartic}
 V_4=\mathcal E_2^{-1}\left(
 Y_2^{[0]}-\frac12\Pi\{S_3,V_3\}\right).
\end{equation}
where
\(
 V_3=ax_1^3+sx_1x_2^2,\, S_3=-L^{-1}V_3.
\)

For \(N=3,4,5\), the corresponding components of \(\mathbf Y\) are
\((Y_N^{[2]},Y_N^{[0]})=(B_{2N}^{[2]}(X),B_{2N}^{[0]}(X))\).
Once the lower-order terms have been determined, the block equation
\eqref{eq:block-matrix} and the invertibility of
\(\widetilde{\mathcal A}_N(a,s)\) determine the zero-edge solution
\begin{equation}\label{eq:finite-zero-edge}
 \begin{pmatrix}\widehat{V_{2N-1}}\\
 \widehat{V_{2N}}\end{pmatrix}
 =\widetilde{\mathcal A}_N(a,s)^{-1}
 \begin{pmatrix}
 Y_N^{[2]}-B_{2N}^{[2]}(V^{<N})\\
 Y_N^{[0]}-B_{2N}^{[0]}(V^{<N}) 
 \end{pmatrix},
 \qquad N=3,4,5.
\end{equation}
  Theorem~\ref{thm:block-affine-line} then gives $(V_{2N-1},V_{2N})$ by adding the normalized generator
$(\Phi_N(a,s),\Psi_N(a,s))$ of \(\ker\mathcal A_N\) with the edge coefficients $a_{1,2N-1}$.

Write $t=a_{14}(X)$, $r=a_{16}(X)$, and $u=a_{18}(X)$. 
The condition \( B_8^{[4]}(X)=\mathsf b_8\) impose
\[p_3(a,s;t)=B_8^{[4]}(V^{<4})-\mathsf b_8=0\]

By the edge equation \eqref{eq:scalar-edge}, the condition
\(\rho_4(B_{10}^{[4]}(X))=\mathsf b_{10}\) imposes
\begin{equation}\label{eq:finite-a16}
 \gamma_4(a,s,t,r)
 :=\mathsf b_{10}-\rho_4B_{10}^{[4]}
 (V^{<4}+\widehat{V_7}+\widehat{V_8})
 -r d_4(a,s,t)=0.
\end{equation}
where
\(
 d_4(a,s,t)
 =d_4(V^{<4})
 =\rho_4\mathcal T_{V^{<4}}(\Phi_4,\Psi_4).
\)

Thus every component of \(\mathcal I_{10}(X)=\mathbf Y\) has been
used, and the resulting jet will be denoted by \(X(a,s,t,r,u)\).

Set
\begin{equation}\label{eq:finite-CD-definition}
 \begin{split}
 \mathsf C(a,s,t,r)&:=B_{12}^{[4]}(V^{<5})
   +\Pi\mathcal T_{V^{<5}}(\widehat{V_9},\widehat{V_{10}}),\\
 \mathsf D(a,s,t,r)&:=\Pi\mathcal T_{V^{<5}}(\Phi_5,\Psi_5).
 \end{split}
\end{equation}
Proposition~\ref{prop:shifted-equation} and the linearity of
\(\mathcal T_{V^{<5}}\) give, 
\begin{equation}\label{eq:global-B12-affine}
 B_{12}^{[4]}(X(a,s,t,r,u))
 =\mathsf C(a,s,t,r)+u\mathsf D(a,s,t,r).
\end{equation}
For $F\in \textrm{span}\{\Omega_1^2,\Omega_1\Omega_2,\Omega_2^2\}$, write
\(
 F=\sum_{\alpha\in\{20,11,02\}}F_\alpha\Omega_1^{2-\alpha_2}\Omega_2^{\alpha_2}.\) Then the condition \(B_{12}^{[4]}(X)=\mathbf Z\) imposes, coefficient by
coefficient,
\begin{equation}\label{eq:global-u-recovery}
 \gamma_{\alpha}(a,s,t,r,u)
 :=Z_{\alpha}-\mathsf C_\alpha(a,s,t,r)-u\mathsf D_\alpha(a,s,t,r)=0,
 \qquad \alpha\in\{20,11,02\}.
\end{equation}


\begin{lemma}\label{lem:rational-compatibility}
Given admissible data \((\mathbf Y,\mathbf Z)\), there exist polynomials
\(P_3,\Gamma_4,\Gamma_{20},\Gamma_{11},\Gamma_{02}\) of $a,s,t,r,u$, with coefficients
determined by the data, and nonnegative integers \(m_3,n_4,n_{\alpha}\) such that
\begin{equation*}
 \begin{aligned}
 p_3(a,s;t)=\frac{P_3(a,s,t)}{a^{m_3}},\quad
 \gamma_4=\frac{\Gamma_4}{a^{n_4}},\quad
 \gamma_{\alpha}=\frac{\Gamma_{\alpha}}{a^{n_{\alpha}}},
 \qquad \alpha\in\{20,11,02\}.
 \end{aligned}
\end{equation*}
These functions have no poles on \(a\neq0\).
\end{lemma}

\begin{proof}
 It is a direct consequence of  Proposition~\ref{prop:rational-affine-dependence}.
\end{proof}

Given an admissible pair \((\mathbf Y,\mathbf Z)\), define its finite
compatibility polynomial system by
\[
\mathcal P_{\mathbf Y,\mathbf Z}(a,s,t,r,u):=\{
\mathsf Q(a,s),P_3(a,s;t),\Gamma_4(a,s,t,r),
\Gamma_{\alpha}(a,s,t,r,u),\, \alpha\in\{20,11,02\}\}.
\]
Denote
\[
 \mathcal S_{\mathbf Y,\mathbf Z}^{\sigma}
 :=\left\{s\in\mathbb R:
 \begin{array}{c}
 \text{there exists }(a,t,r,u)\in\mathbb R^4\text{ such that}\\
 \mathcal P_{\mathbf Y,\mathbf Z}(a,s,t,r,u)=0
 \text{ and }\sigma a>0
 \end{array}
 \right\}.
\]

\begin{lemma}\label{prop:finite-back-substitution}
Fix admissible data \((\mathbf Y,\mathbf Z)\) and a sign
\(\sigma\in \{\pm 1\}\).  There is a bijection between the real common
zeros \((a,s,t,r,u)\) of
\(\mathcal P_{\mathbf Y,\mathbf Z}\) satisfying \(\sigma a>0\) and the
ten-jets \(X\)
satisfying
\begin{equation}\label{eq:finite-back-substitution}
 \mathcal I_{10}(X)=\mathbf Y,
 \qquad
 B_{12}^{[4]}(X)=\mathbf Z,
 \qquad
 \sigma a_{30}(X)>0,
\end{equation}
where
\(
 (a,s,t,r,u)
 =(a_{30}(X),a_{12}(X),a_{14}(X),a_{16}(X),a_{18}(X)).
\)
\end{lemma}



\begin{proposition}\label{prop:recover-a12}
Suppose an admissible data \((\mathbf Y,\mathbf Z)\) is realized by
\(X_*\), and assume
\(\sigma a_{30}(X_*)>0\).  Then
\(
 \mathcal S_{\mathbf Y,\mathbf Z}^{\sigma}
 =\{a_{12}(X_*)\}.
\)
Consequently, the data determine \(a_{12}(X_*)\) uniquely.
\end{proposition}

\begin{proof}
The coordinates of \(X_*\) give a real common zero of
\(\mathcal P_{\mathbf Y,\mathbf Z}\), so
\(a_{12}(X_*)\in\mathcal S_{\mathbf Y,\mathbf Z}^{\sigma}\).
Conversely, if \(s\in\mathcal S_{\mathbf Y,\mathbf Z}^{\sigma}\),
choose a corresponding \((a,t,r,u)\), they
produce a ten-jet \(X\) satisfying \eqref{eq:same-finite-data} and
\(\sigma a_{30}(X)>0\). By Proposition~\ref{prop:separate-cubic}, one get
\((a_{30}(X),a_{12}(X))=\pm(a_{30}(X_*),a_{12}(X_*))\).  The sign
condition excludes the minus alternative,  hence
\(s=a_{12}(X_*)\).
\end{proof}

\subsection{Proof of the main theorem}

\begin{proof}[Proof of Theorem~\ref{thm:main}]
Define
\[\mathfrak G_{10,v}^{\sigma}=\{X\in J_{10,v}^{\sigma}: a_{12}(X)\neq 0, \kappa_2(X)\neq 0, \chi^\sharp(X)\neq0\}\cap \bigcap_{N\geq 4}\{d_N(X)\neq0\}\cap \mathcal U_{\mathrm{sep}}\]
It is well defined since the quantitives defining the conditions depends only on $j^{10}V$.

By Proposition~\ref{prop:separate-cubic}, the set $\mathcal U_{\mathrm{sep}}$ is a dense $G_{\delta}$ subset in $J_{10}$. 
Combined with Proposition~\ref{prop:normalized-compatibility-generic} and Proposition~\ref{prop:d-nonzero}, we get that $\mathfrak G_{10,v}^{\sigma}$ is a dense $G_{\delta}$ subset in $J_{10,v}^{\sigma}$. 
Since the projection map $j^{10}$ is continuous and open, the set $\mathcal G_v^{\sigma}=(j^{10})^{-1}\mathfrak G_{10,v}^{\sigma}$ is a dense $G_{\delta}$ set in $\mathcal V_v^{\sigma}$. 

Now let \(V\in\mathcal G_v^\sigma\), and let
\(W\in\mathcal V_v^\sigma\) have the same
\((B^{[0]},B^{[2]},B^{[4]})\).  Their ten-jets have the same
\(\mathcal I_{10}\)-datum and the same \(B_{12}^{[4]}\); hence
Proposition~\ref{prop:separate-cubic} gives
\(
 (a_{30}(W),a_{12}(W))
 =(a_{30}(V),a_{12}(V)).
\)
since their \(a_{30}\)-coefficients have
the same sign $\sigma$. Applying
Theorem~\ref{thm:recover-a14}, we get \(W=V\).

For the determination of the Taylor coefficients, extract
\((\mathbf Y,\mathbf Z)\) from the QBNF data and form
\(\mathcal P_{\mathbf Y,\mathbf Z}\).  Proposition~\ref{prop:recover-a12}
determines \(s=a_{12}(V)\) uniquely.
The equation \(\mathsf Q=0\), together with the prescribed sign $\sigma$, then gives $a_{30}$.
With \(a_{30}\) and \(a_{12}\) now determined,
Theorem~\ref{thm:recover-a14} determines \(a_{14}\) and then the remaining
homogeneous terms successively.
\end{proof}

\begin{proof}[Proof of Theorem~\ref{thm:main-mod-symmetry}]
Let \(\mathfrak G_{10,v}^+\) be the set given by
Theorem~\ref{thm:main} for \(\sigma=+1\), and set
\[
 \mathfrak G_{10,v}
 :=\mathfrak G_{10,v}^+\cup\iota(\mathfrak G_{10,v}^+).
\]
This is an \(\iota\)-invariant dense \(G_\delta\) set in \(J_{10,v}\).
If \(V,W\in\mathcal V_v\) have the same first three QBNF layers and
\(j^{10}V\in\mathfrak G_{10,v}\), by the \(\iota\)-invariance of the
QBNF, we can assume that \(j^{10}V\in\mathfrak G_{10,v}^+\).
By Proposition~\ref{prop:separate-cubic}, either \(W\) or \(\iota W\)
lies in \(\mathcal V_v^+\).  Theorem~\ref{thm:main} then gives
\(W=V\) or \(W=\iota V\), and its determination assertion gives the
remaining conclusions.
\end{proof}

\begin{proof}[Proof of Theorem~\ref{thm:analytic-genericity}]
Let \(\mathfrak G_{10,v}^\sigma\) be the dense \(G_\delta\) subset
constructed in the proof of Theorem~\ref{thm:main}, and set
\(
 \mathcal G_v^\sigma
 :=(j^{10})^{-1}(\mathfrak G_{10,v}^\sigma).
\)
By Proposition~\ref{prop:analytic-jet-topology}, \(\mathcal G_v^\sigma\)
is a dense \(G_\delta\) subset of \(\mathcal A_v^\sigma\).

If \(V\in\mathcal G_v^\sigma\) and \(W\in\mathcal A_v^\sigma\) have
the same first three QBNF layers, then Theorem~\ref{thm:main} identifies
their Taylor series. Since \(V\) is real analytic on the connected domain \(\mathbb R^2\), this Taylor series determines \(V\) globally.
\end{proof}

\begin{proof}[Proof of Theorem~\ref{thm:main-spectral}]
Set
\(
 \mathcal G_v^{\mathrm{trap}}
 :=(j^{10})^{-1}(\mathfrak G_{10,v}).
\)
Proposition~\ref{prop:analytic-trapping-jet-topology} shows that
\(\mathcal G_v^{\mathrm{trap}}\) is a dense \(G_\delta\) subset of
\(\mathcal A_v^{\mathrm{trap}}\).  It is \(\iota\)-invariant because
\(\mathfrak G_{10,v}\) is \(\iota\)-invariant and \(j^{10}\) commutes
with \(\iota\).

Let \(V\in\mathcal G_v^{\mathrm{trap}}\) and
\(W\in\mathcal A_v^{\mathrm{trap}}\), and suppose that
\(
 \operatorname{Spec}_{\mathrm{sc}}(P_\hbar^V)
 =\operatorname{Spec}_{\mathrm{sc}}(P_\hbar^W).
\)
The spectral determination of the QBNF implies that \(V\) and \(W\) have
the same complete QBNF, we conclude by applying the sign-free form of
Theorem~\ref{thm:analytic-genericity}, similar to Theorem \ref{thm:main-mod-symmetry}.

Finally, if
\(
 \operatorname{Spec}_{\delta}(P_\hbar^V)
 =\operatorname{Spec}_{\delta}(P_\hbar^W)
\)
for all sufficiently small \(\hbar>0\), it implies equality of the semiclassical spectra, and the same
conclusions follow.
\end{proof}
\bigskip

\appendix

\section{Genericity via finite-jet maps}
\label{app:analytic-topology}






We use the standard holomorphic-germ topology described in
Kriegl--Michor~\cite[Sections~3.1, 3.3, and 3.10--3.12]{KM90}.  Let
\(\mathcal H(\R^2\subset\C^2)\) be the space of holomorphic germs along
\(\R^2\), and let \(\mathcal N\) be the collection of the complex
neighborhoods of $\mathbb R^2$.  For each \(W\in\mathcal N\), equip
\(\mathcal O(W)\) with the compact-open topology and let
\[
 q_W:\mathcal O(W)\longrightarrow\mathcal H(\R^2\subset\C^2)
\]
be the germ map.  The germ space carries the finest locally convex topology
for which every \(q_W\) is continuous.  In particular, compact-open convergence on a fixed \(W\)
implies convergence of the corresponding germs.

We identify \(C^\omega(\R^2;\R)\) with the conjugation-invariant real
subspace of \(\mathcal H(\R^2\subset\C^2)\) and give it the relative
topology.  Let \(\mathcal A_v\) be the subspace consisting of potentials
whose Taylor series at the origin has the form \eqref{eq:Vm-expansion}, and set
\(
 \mathcal A_v^\sigma
 :=\{V\in\mathcal A_v:\sigma a_{30}(V)>0\},
 \, \sigma\in\{\pm1\}.
\)
Both spaces carry the relative topology just described.

\begin{proposition}\label{prop:analytic-jet-topology}
For every \(m\ge3\) and \(\sigma\in\{\pm1\}\), the jet map
\(
 j^m:\mathcal A_v^\sigma\longrightarrow J_{m,v}^\sigma
\)
is continuous and open. Consequently, if
\(\mathfrak G\subset J_{m,v}^\sigma\) is dense \(G_\delta\), then
\((j^m)^{-1}(\mathfrak G)\) is dense \(G_\delta\) in
\(\mathcal A_v^\sigma\).
\end{proposition}

\begin{proof}
For every multi-index \(\alpha\) and every \(W\in\mathcal N\), the Cauchy
estimate on a closed polydisc \(\overline D_r\Subset W\) gives
\(
 |\partial^\alpha F(0)|
 \le \alpha!r^{-|\alpha|}
 \sup_{Z\in\overline D_r}|F(Z)|.
\)
Hence derivative evaluation, and therefore \(j^m\), is continuous.  Write
\(
 J^\circ_{m,v}:=J_{m,v}-V_2
\)
for the vector space obtained by removing the fixed quadratic part. Define
\(
 s_m:J^\circ_{m,v}\longrightarrow C^\omega(\R^2;\R),
 \, s_m(Y):=Y_3+\cdots+Y_m.
\)
The map \(s_m\) is continuous: coefficient convergence in the finite
dimensional space \(J^\circ_{m,v}\) gives compact-open convergence of the
corresponding polynomials in \(\mathcal O(\C^2)\), hence, by the continuity
of \(q_W\) with \(W=\C^2\), convergence of their germs in
\(C^\omega(\R^2;\R)\).  Moreover,
\(
 j^m\bigl(V+s_m(Y)\bigr)=j^mV+Y.
\)
 If \(U\subset\mathcal A_v^\sigma\) is open
and \(V\in U\), 
there is a neighborhood \(B\) of zero in \(J^\circ_{m,v}\) such that
\(V+s_m(B)\subset U\).  Hence
\(
 j^mV+B=j^m\bigl(V+s_m(B)\bigr)\subset j^m(U).
\)
Thus \(j^mV\) is an interior point of \(j^m(U)\).  Thus \(j^m\) is open.

Finally, let \(\mathfrak G\subset J_{m,v}\) be dense \(G_\delta\). Continuity implies that \((j^m)^{-1}(\mathfrak G)\) is \(G_\delta\). For every nonempty open set \(U\subset\mathcal A_v^{\mathrm{trap}}\), the image \(j^m(U)\) is nonempty and open, so it meets \(\mathfrak G\). Hence \(U\) meets \((j^m)^{-1}(\mathfrak G)\), proving density.

\end{proof}

Let \(\mathcal A_v^{\mathrm{trap}}\subset\mathcal A_v\) be the class of
potentials for which the origin is the unique global minimum and for which
\(V^{-1}([0,\varepsilon])\) is compact for some \(\varepsilon>0\).  We give
\(\mathcal A_v^{\mathrm{trap}}\) the relative topology inherited from
\(\mathcal A_v\).

\begin{proposition}\label{prop:analytic-trapping-jet-topology}
For every \(m\ge3\), the restricted jet map
\(
 j^m:\mathcal A_v^{\mathrm{trap}}\longrightarrow J_{m,v}
\)
is continuous and open.  Consequently, if \(\mathfrak G\subset J_{m,v}\) is dense
\(G_\delta\), then \((j^m)^{-1}(\mathfrak G)\) is dense \(G_\delta\) in
\(\mathcal A_v^{\mathrm{trap}}\).
\end{proposition}

\begin{proof}
Cauchy estimates imply that derivative evaluation at the origin is continuous in the holomorphic-germ topology. Hence the restricted jet map \(j^m\) is continuous.

Write \(J_{m,v}^{\circ}:=J_{m,v}-V_2\), and identify each \(Y=(Y_3,\ldots,Y_m)\in J_{m,v}^{\circ}\) with the polynomial \(Y(x)=Y_3(x)+\cdots+Y_m(x)\). Define
\[s_m^{\mathrm{dec}}(Y)(x):=e^{-|x|^{2m}}Y(x),\qquad |x|^{2m}=(x_1^2+x_2^2)^m.\]
This function is real analytic and has the required reflection symmetry. The map \(s_m^{\mathrm{dec}}:J_{m,v}^{\circ}\to C^\omega(\mathbb R^2,\mathbb R)\) is continuous, being linear with finite-dimensional domain. Since \(e^{-|x|^{2m}}=1+O(|x|^{2m})\) and \(Y\) vanishes to order at least three, the Taylor polynomial of \(s_m^{\mathrm{dec}}(Y)\) through degree \(m\) is \(Y\). Consequently,
\(j^m\bigl(V+s_m^{\mathrm{dec}}(Y)\bigr)=j^mV+Y.\)
Fix a norm on \(J_{m,v}^{\circ}\), and set \(w(x):=\min\{|x|^2,1\}\). There is a constant \(C_m>0\) such that
\[|s_m^{\mathrm{dec}}(Y)(x)|\le C_m\|Y\|\,w(x)\qquad\text{for all }x\in\mathbb R^2.\]
Indeed, the coefficients of \(Y\) are bounded by a constant times \(\|Y\|\). For \(|x|\le1\), its vanishing to order at least three gives \(|Y(x)|\le C_m\|Y\||x|^3\le C_m\|Y\||x|^2\). For \(|x|\ge1\), each monomial of degree at most \(m\), multiplied by \(e^{-|x|^{2m}}\), is bounded.

Fix \(V\in\mathcal A_v^{\mathrm{trap}}\). We claim that there exists \(c_V>0\) such that
\[V(x)\ge c_Vw(x)\qquad\text{for all }x\in\mathbb R^2.\]
Near the origin, this follows from the positive definite quadratic part of \(V\). Outside a fixed small ball, \(V\) has a positive lower bound. Otherwise, a sequence outside that ball with \(V(x_j)\to0\) would eventually lie in a compact low sublevel set and have a subsequence converging to a point distinct from the origin at which \(V\) vanishes, contradicting uniqueness of the global minimum.

It follows that, whenever \(\|Y\|<c_V/(2C_m)\),
\[|s_m^{\mathrm{dec}}(Y)(x)|\le\tfrac12V(x),\qquad V(x)+s_m^{\mathrm{dec}}(Y)(x)\ge\tfrac12V(x).\]
Thus the origin remains the unique global minimum, and the quadratic part is unchanged. Moreover, if \(V^{-1}([0,\varepsilon])\) is compact, so is the subset
\[\bigl(V+s_m^{\mathrm{dec}}(Y)\bigr)^{-1}([0,\varepsilon/2])\subset V^{-1}([0,\varepsilon]).\]
Hence \(V+s_m^{\mathrm{dec}}(Y)\in\mathcal A_v^{\mathrm{trap}}\) for all sufficiently small \(Y\).

Now let \(U\subset\mathcal A_v^{\mathrm{trap}}\) be open and let \(V\in U\). By continuity of \(Y\mapsto V+s_m^{\mathrm{dec}}(Y)\) and the preceding trapping property, there is an open neighborhood \(B\subset J_{m,v}^{\circ}\) of zero such that \(V+s_m^{\mathrm{dec}}(B)\subset U\). Therefore,
\(j^mV+B=j^m\bigl(V+s_m^{\mathrm{dec}}(B)\bigr)\subset j^m(U).\)
Thus every point of \(j^m(U)\) is interior, proving that \(j^m\) is open.

\end{proof}

\section{QBNF invariance and Moyal identities}\label{app:moyal-identities}
\subsection{
 Uniqueness of QBNF with fixed gauge}
\begin{proof}[Proof of uniqueness in Theorem~\ref{thm:qbnf-construction}]
Suppose
\(
        \exp(D_S)H=B,\, \exp(D_{\widetilde S})H=\widetilde B
\)
are two formal normalizations, with \(B\) and \(\widetilde B\) resonant.  Since
\(
        \widetilde B
        =\exp(D_{\widetilde S})\exp(-D_S)B,
\)
and the derivations \(D_A\) form a filtered Lie algebra,
\([D_A,D_C]=D_{D_A C}\), the BCH logarithm of
\(\exp(D_{\widetilde S})\exp(-D_S)\) is well defined formally.  Thus there is
a formal generator \(U=U_3+U_4+\cdots\) such that
\(
        \widetilde B=\exp(D_U)B.
\)
Write \(U_m=U_m^{\rm res}+U_m^\perp\) according to the splitting
\eqref{eq:splitting}.  If all \(U_m^\perp\) vanish, then \(U\) is resonant.
Lemma~\ref{lem:resonant-moyal} shows that the resonant algebra is
Moyal-commutative.  Hence \(D_U B=0\), and \(\widetilde B=B\).

It remains to exclude the first non-resonant component.  Let \(m\) be the
smallest degree for which \(U_m^\perp\ne0\).  All lower degree components of \(U\) are
resonant and therefore commute with \(B\).  In degree \(m\), the only term in
\(\exp(D_U)B\) containing \(U_m^\perp\) is
\(
        D_{U_m^\perp}H_2=L(U_m^\perp).
\)
All other terms involving \(U_m^\perp\) have degree \(>m\).  Thus
\(
        [\widetilde B-B]_m=L(U_m^\perp).
\)
The left-hand side is resonant, while the right-hand side is non-resonant and
nonzero since \(L\) is invertible on the non-resonant subspace.  This
contradiction proves that no such \(m\) exists.  Therefore the resonant normal
form \(B\) is unique.
\end{proof}

\subsection{ Invariance of QBNF}
\begin{proposition}\label{prop:qbnf-parity}
Assume \(v_1/v_2\notin\mathbb Q\).  Let
\(
 V(x)=\sum_{|\alpha|\ge3}a_\alpha x^\alpha
\)
be an arbitrary real formal potential, and set
\[
 H_V(x,\xi):=H_2(x,\xi)+V(x),
 \qquad
 H_2(x,\xi):=\frac12\bigl(v_1(x_1^2+\xi_1^2)
                         +v_2(x_2^2+\xi_2^2)\bigr).
\]
Denote \(V^-(x):=V(-x)\), then
\(
 \mathcal B(V^-)=\mathcal B(V).
\)
\end{proposition}

\begin{proof}
Let
\(
 \iota(x,\xi)=(-x,-\xi),\) and \(
 \,
 (\iota^*A)(x,\xi;\hbar):=A(-x,-\xi;\hbar).
\)
For every multi-index \(\gamma\), the chain rule gives
\[
 \partial^\gamma(\iota^*A)
 =(-1)^{|\gamma|}\iota^*(\partial^\gamma A).
\]
Each term in \(\{A,B\}_m\) contains exactly \(m\) derivatives of \(A\)
and \(m\) derivatives of \(B\).  The two signs therefore cancel, and
\(
 \{\iota^*A,\iota^*B\}_m
 =\iota^*\{A,B\}_m.
\)
The Moyal expansion implies
\(
 \iota^*(A*B)=(\iota^*A)*(\iota^*B),
\)
and hence
\begin{equation}\label{eq:iota-derivation-covariance}
 \iota^*(D_SA)=D_{\iota^*S}(\iota^*A),
 \qquad
 \iota^*\exp(D_S)A
 =\exp(D_{\iota^*S})(\iota^*A).
\end{equation}

The involution fixes \(H_2\), \(\hbar\), and the actions
\(\Omega_1,\Omega_2\).  In complex coordinates,
\[
 \iota^*(\hbar^r z^\alpha\bar z^\beta)
 =(-1)^{|\alpha|+|\beta|}
  \hbar^r z^\alpha\bar z^\beta.
\]
Therefore
\(
 \Pi\iota^*=\iota^*\Pi,
 \,
 \iota^*|_{\operatorname{range}\Pi}=\operatorname{Id}.
\)
Moreover, \(L\iota^*=\iota^*L\).  Since
\(\iota^*\) preserves the nonresonant subspace, uniqueness of the
nonresonant solution of the homological equation gives
\(
 L^{-1}\iota^*=\iota^*L^{-1}.
\)
Thus the gauge-fixed QBNF recursion is equivariant under \(\iota^*\).

Let \(S\) be the gauge-fixed generator in
Theorem~\ref{thm:qbnf-construction}.  Applying \(\iota^*\) to
\[
 \exp(D_S)H_V=\mathcal B(V), \quad \Pi S=0,
\]
 we obtain
\[
 \exp(D_{\iota^*S})H_{V^-}=\mathcal B(V),
 \qquad \Pi(\iota^*S)=0.
\]
Here we used that \(\iota^*\) is the identity on the resonant algebra. We conclude by the 
uniqueness of QBNF in Theorem~\ref{thm:qbnf-construction} .

\end{proof}




\subsection{Some Moyal product identities}
\begin{lemma}\label{lem:moyal-associativity}
The Moyal product \eqref{eq:moyal-product} is associative.
\end{lemma}

\begin{proof}
Let \(X^{(r)}=(x^{(r)},\xi^{(r)})\), \(r=1,2,3\), be independent copies
of the coordinates.  Set
\[
 \Lambda_{rs}:=\sum_{j=1}^2
 \left(
 \partial_{\xi_j^{(r)}}\partial_{x_j^{(s)}}
 -\partial_{x_j^{(r)}}\partial_{\xi_j^{(s)}}
 \right),
 \qquad 1\leq r<s\leq3,
\]
and, for a symbol \(F=F(X^{(1)},X^{(2)},X^{(3)})\), define its
restriction to the full diagonal by
\(
 (m_3F)(X):=F(X,X,X).
\)
Writing \(A_r=A(X^{(r)})\), and similarly for \(B_r,C_r\), the chain rule
shows that both \((A*B)*C\) and \(A*(B*C)\) are equal to
\begin{align*}
 m_3\exp\left(\frac{\hbar}{2i}
   (\Lambda_{12}+\Lambda_{13}+\Lambda_{23})\right)A_1B_2C_3.
\end{align*}
Indeed, the three constant-coefficient operators \(\Lambda_{rs}\) commute
pairwise.  This proves the claim. 
\end{proof}


\begin{lemma}\label{lem:resonant-moyal}
Let \(C\in\cW\) be resonant.  Then, for every \(A\in\cW\),
\begin{equation}\label{eq:average-star}
 \Pi(A*C)=(\Pi A)*C,\qquad
 \Pi(C*A)=C*(\Pi A).
\end{equation}
The resonant subalgebra is commutative for the Moyal product.  Consequently,
\(
\Pi D_A C=0.
\)
\end{lemma}

\begin{proof}
We define the rotations used for the resonant projection.  For
\(\theta=(\theta_1,\theta_2)\in\mathbb T^2\), let \(R_\theta\) act on a
symbol \(A\) by
\begin{equation}\label{eq:appendix-Rtheta-definition}
 (R_\theta A)(x,\xi):=A(x^\theta,\xi^\theta),\qquad
 \begin{pmatrix}x_j^\theta\\ \xi_j^\theta\end{pmatrix}
 =
 \begin{pmatrix}
  \cos\theta_j&-\sin\theta_j\\
  \sin\theta_j& \cos\theta_j
 \end{pmatrix}
 \begin{pmatrix}x_j\\ \xi_j\end{pmatrix},
 \quad j=1,2.
\end{equation}
Thus \(R_\theta\) is the pullback by a real linear phase-space rotation.
In particular,
\[R_\theta z_j=e^{i\theta_j}z_j,
\qquad
R_\theta\bar z_j=e^{-i\theta_j}\bar z_j.\]


For completeness, let \(u_j=(x_j,\xi_j)\), \(v_j=(y_j,\eta_j)\), and
\(J=\left(\begin{smallmatrix}0&-1\\1&0\end{smallmatrix}\right)\).  Then
\[
 \Lambda(AB)=\sum_{j=1}^2
 (\nabla_{u_j}A)^{\mathsf T}J\nabla_{v_j}B.
\]
Writing \(Q_{\theta_j}\) for the rotation matrix in
\eqref{eq:appendix-Rtheta-definition}, we have
\(Q_{\theta_j}^{\mathsf T}JQ_{\theta_j}=J\); hence the chain rule gives
\[
 \Lambda\bigl((R_\theta A)(x,\xi)(R_\theta B)(y,\eta)\bigr)
 =
 R_\theta^{(1)}R_\theta^{(2)}
 \Lambda\bigl(A(x,\xi)B(y,\eta)\bigr),
\]
where \(R_\theta^{(1)}\) rotates only \((x,\xi)\), while
\(R_\theta^{(2)}\) rotates only \((y,\eta)\).
The same identity holds for every power of \(\Lambda\).  Substituting it
into \eqref{eq:moyal-product} and restricting to the diagonal proves
\begin{equation}\label{eq:appendix-moyal-covariance}
 R_\theta(A*B)=(R_\theta A)*(R_\theta B).
\end{equation}

For a monomial one has
\[
 R_\theta
 \bigl(\hbar^r z^\alpha\bar z^\beta\bigr)
 =
 e^{i\langle\theta,\alpha-\beta\rangle}
 \hbar^r z^\alpha\bar z^\beta.
\]
Averaging this identity over \(\mathbb T^2\) gives zero unless
\(\alpha=\beta\).  Since \(v_1/v_2\notin\mathbb Q\), these are precisely
the resonant monomials.  Therefore $\Pi$ is the torus average of the rotations \(R_\theta\):
\begin{equation}\label{eq:appendix-Pi-average}
 \Pi A=\frac1{(2\pi)^2}
 \int_{\mathbb T^2}R_\theta A\,\dd\theta.
\end{equation}
If \(C\) is resonant, then \(R_\theta C=C\).  Combining
\eqref{eq:appendix-moyal-covariance} and
\eqref{eq:appendix-Pi-average} gives \eqref{eq:average-star}.

It remains to prove that the resonant algebra is commutative.  Every resonant symbol is a formal series in
\(\hbar,\Omega_1,\Omega_2\). Notice that,
\(
 \Omega_1*\Omega_2=\Omega_1\Omega_2=\Omega_2*\Omega_1
\)
because the two actions involve disjoint variables.  For a single action
\(\Omega=x^2+\xi^2\), the Moyal formula gives, for \(k\geq1\),
\[
 \Omega*\Omega^k
 =\Omega^k*\Omega
 =\Omega^{k+1}-\hbar^2k^2\Omega^{k-1}.
\]
This recursion shows inductively that every ordinary power of \(\Omega\)
is a polynomial in \(\Omega\) under the Moyal product.  Thus every
resonant symbol is a formal Moyal series in the commuting elements
\(\Omega_1,\Omega_2\), and the resonant algebra is commutative.

Finally, \eqref{eq:average-star} gives
\[
 \Pi[A,C]_* =(\Pi A)*C-C*(\Pi A) =[\Pi A,C]_* =0.
\]
Multiplication by \(i/\hbar\) gives \(\Pi D_A C=0\), completing the
proof.
\end{proof}

\begin{lemma}\label{lem:higher-bracket-Leibniz}
Let \(A\) and \(B\) be \(\hbar\)-independent symbols.  For every \(q\geq1\),
\begin{equation}\label{eq:higher-bracket-Leibniz}
 L\{A,B\}_q
 =
 \{L(A),B\}_q+\{A,L(B)\}_q.
\end{equation}
\end{lemma}

\begin{proof}
Associativity of the Moyal product gives
\begin{equation}\label{eq:appendix-L-star-derivation}
 L(A*B)=L(A)*B+A*L(B).
\end{equation}
Comparing the coefficients of $\hbar^q$ on the two sides of
\eqref{eq:appendix-L-star-derivation}, using the Moyal expansion
\eqref{eq:moyal-product}, proves \eqref{eq:higher-bracket-Leibniz}.
\end{proof}

\begin{lemma}\label{lem:projected-bracket-identity}
Let \(N\geq3\) and \(F\in\mathsf O_{2N-1}\).  For odd \(q\geq 1\), one has
\begin{equation}\label{eq:projected-bracket-identity}
        \Pi\{-L^{-1}F,V_3\}_q=\Pi\{S_3,F\}_q.
\end{equation}
\end{lemma}

\begin{proof}
 Set \(T=-L^{-1}F\).  Then
\(L(S_3)=-V_3\) and \(L(T)=-F\), by
Lemma~\ref{lem:higher-bracket-Leibniz} and the skew-symmetry of the
odd brackets give
\[
\begin{aligned}
 L\{S_3,T\}_q
 &=\{L(S_3),T\}_q+\{S_3,L(T)\}_q\\
 &=\{T,V_3\}_q-\{S_3,F\}_q.
\end{aligned}
\]
Applying \(\Pi\) and using \(\Pi L=0\), we obtain
\(
 \Pi\{T,V_3\}_q=\Pi\{S_3,F\}_q.
\)
Substitution of \(T=-L^{-1}F\) proves
\eqref{eq:projected-bracket-identity}.
\end{proof}

\end{document}